\documentclass[twoside,english,american,DIV12,listtotoc,bibtotoc,idxtotoc]{scrartcl}
\usepackage[T1]{fontenc}
\usepackage[latin1]{inputenc}
\usepackage{babel}
\usepackage{prettyref}
\usepackage{mathtools}
\usepackage{amsthm}
\usepackage{amsmath}
\usepackage{amssymb}
\usepackage[unicode=true,pdfusetitle,
 bookmarks=true,bookmarksnumbered=false,bookmarksopen=false,
 breaklinks=false,pdfborder={0 0 1},backref=false,colorlinks=false]
 {hyperref}

\makeatletter
 \theoremstyle{definition}
 \newtheorem*{defn*}{\protect\definitionname}
\theoremstyle{plain}
\newtheorem{thm}{\protect\theoremname}[section]
  \theoremstyle{remark}
  \newtheorem{rem}[thm]{\protect\remarkname}
  \theoremstyle{plain}
  \newtheorem{prop}[thm]{\protect\propositionname}
  \theoremstyle{definition}
  \newtheorem{example}[thm]{\protect\examplename}
  \theoremstyle{plain}
  \newtheorem{lem}[thm]{\protect\lemmaname}
  \theoremstyle{plain}
  \newtheorem{cor}[thm]{\protect\corollaryname}

\usepackage{helvet}
\usepackage[T1]{fontenc}

\usepackage{amsmath}
\usepackage{setspace}
\usepackage{prettyref}
\usepackage{enumerate}
\usepackage{url}

\makeatletter

\usepackage[T1]{fontenc}    % Silbentrennung auch nach Umlauten m?glich

\usepackage{a4wide}        % A4-Format
\usepackage{tu-preprint}
\usepackage{amsmath}

\usepackage{euscript}
\allowdisplaybreaks[1] % allow page breaks in some displayed equations
\usepackage{amsfonts}
\newenvironment{keywords}{ \noindent\footnotesize\textbf{Keywords and phrases:}}{}

\newenvironment{class}{\noindent\footnotesize\textbf{Mathematics subject classification 2010:}}{}

\usepackage{color,tu-preprint}

\newcommand*{\trace}{\operatorname{trace}}

\newcommand*{\dive}{\operatorname{div}}

\newcommand*{\Grad}{\operatorname{Grad}}
\newcommand*{\Div}{\operatorname{Div}}

\newcommand*{\grad}{\operatorname{grad}}

\renewcommand*{\i}{\mathrm{i}}
\DeclareMathAccent{\Circ}{\mathalpha}{operators}{"17}

\renewcommand{\Re}{\operatorname{\mathfrak{Re}}}

\newcommand{\lspan}{\operatorname{span}}

\renewcommand{\tilde}{\widetilde}
\renewcommand*{\epsilon}{\varepsilon}
\renewcommand*{\rho}{\varrho}

\DeclareMathOperator*{\wlim}{\mathrm{w-lim}}

\newrefformat{prop}{Proposition \ref{#1}}
\newrefformat{lem}{Lemma \ref{#1}}
\newrefformat{thm}{Theorem \ref{#1}}
\newrefformat{cor}{Corollary \ref{#1}}
\newrefformat{rem}{Remark \ref{#1}}
\newrefformat{eq}{(\ref{#1})}
\newrefformat{item}{(\ref{#1})}

\makeatother

\usepackage{babel}
\institut{Institut f\"ur Analysis}

\preprintnumber{MATH-AN-11-2013}

\preprinttitle{A characterization of boundary conditions yielding maximal monotone operators.}

\author{Sascha Trostorff}

\AtBeginDocument{
  
}

\makeatother

  \addto\captionsamerican{\renewcommand{\corollaryname}{Corollary}}
  \addto\captionsamerican{\renewcommand{\definitionname}{Definition}}
  \addto\captionsamerican{\renewcommand{\examplename}{Example}}
  \addto\captionsamerican{\renewcommand{\lemmaname}{Lemma}}
  \addto\captionsamerican{\renewcommand{\propositionname}{Proposition}}
  \addto\captionsamerican{\renewcommand{\remarkname}{Remark}}
  \addto\captionsamerican{\renewcommand{\theoremname}{Theorem}}
  \addto\captionsenglish{\renewcommand{\corollaryname}{Corollary}}
  \addto\captionsenglish{\renewcommand{\definitionname}{Definition}}
  \addto\captionsenglish{\renewcommand{\examplename}{Example}}
  \addto\captionsenglish{\renewcommand{\lemmaname}{Lemma}}
  \addto\captionsenglish{\renewcommand{\propositionname}{Proposition}}
  \addto\captionsenglish{\renewcommand{\remarkname}{Remark}}
  \addto\captionsenglish{\renewcommand{\theoremname}{Theorem}}
  \providecommand{\corollaryname}{Corollary}
  \providecommand{\definitionname}{Definition}
  \providecommand{\examplename}{Example}
  \providecommand{\lemmaname}{Lemma}
  \providecommand{\propositionname}{Proposition}
  \providecommand{\remarkname}{Remark}
\providecommand{\theoremname}{Theorem}

\begin{document}
\makepreprinttitlepage

\author{ Sascha Trostorff \\ Institut f\"ur Analysis, Fachrichtung Mathematik\\ Technische Universit\"at Dresden\\ Germany\\ sascha.trostorff@tu-dresden.de}

\title{A characterization of boundary conditions yielding maximal monotone
operators.}

\maketitle
\begin{abstract} \textbf{Abstract.} We provide a characterization
for maximal monotone realizations for a certain class of (nonlinear)
operators in terms of their corresponding boundary data spaces. The
operators under consideration naturally arise in the study of evolutionary
problems in mathematical physics. We apply our abstract characterization
result to Port-Hamiltonian systems and a class of frictional boundary
conditions in the theory of contact problems in visco-elasticity.\end{abstract}

\begin{keywords} Maximal monotone operators, boundary data spaces,
nonlinear boundary conditions, Port-Hamiltonian systems, frictional
boundary conditions\end{keywords}

\begin{class} 47B44,47F05,47N20,46N20  \end{class}

\newpage

\tableofcontents{} 

\newpage

\section{Introduction}

As it was shown in several articles (\cite{Picard,Picard2012_Impedance,Picard2012_mother,Picard2013_nonauto,Picard2013_fractional,Trostorff2012_NA,Trostorff2012_nonlin_bd,Trostorff2012_integro,Trostorff2013_nonautoincl})
evolutionary problems in classical mathematical physics can often
be written as a differential equation of the form
\[
\left(\partial_{0}\mathcal{M}+A\right)u=f,
\]
where $u$ is the unknown, $f$ is a given source term, $\partial_{0}$
denotes the temporal derivative, $\mathcal{M}$ is a suitable bounded
operator acting in space-time and $A$ is a maximal monotone (possibly
nonlinear) operator, which frequently is a suitable restriction of
a block operator matrix of the form 
\begin{equation}
\left(\begin{array}{cc}
0 & D\\
G & 0
\end{array}\right),\label{eq:block}
\end{equation}
where $G$ and $D$ are densely defined closed linear operators satisfying
$-G^{\ast}\subseteq D$. 

The aim of this article is to provide a characterization of all maximal
monotone restrictions of \prettyref{eq:block}. This characterization
will be given in terms of the so-called boundary data spaces, introduced
in \cite{Picard2012_comprehensive_control,Picard2012_boundary_control},
associated with the operators $G$ and $D$. Moreover, we give a characterization
of skew-selfadjoint (and hence maximal monotone) restrictions of \prettyref{eq:block},
which is a natural question arising for instance in the study of energy
preserving evolutionary problems (see e.g. \cite{Picard2012_conservative,Picard2012_comprehensive_control}). 

The question of maximal monotone (or m-accretive) realizations of
certain operators or relations was studied in various papers. For
instance in 1959, Phillips \cite{Phillips1959} provides a characterization
of m-accretive realizations of linear operators using indefinite metrics
on Hilbert spaces on the one hand and the Neumann-Cayley transform
on the other hand. Later on these results were generalized to linear
relations in \cite{Dijskma1974}. More recently, in \cite{Arlinskii1995}
we find a characterization result for m-accretive extensions of linear
relations in Hilbert spaces using the theory of Friedrichs- and Neumann-extensions
of symmetric relations \cite{Coddington1978}. Another strategy to
study extensions of operators or relations uses the theory of boundary
triplets or, more general, boundary relations (see e.g. \cite{Derkach2006,Behrndt2009,Derkach2009}).
So, for instance in \cite[Chapter 3]{Hassi2012} the question of m-accretive
extensions of sectorial operators is addressed and a characterization
is given in terms of boundary triplets. To the authors best knowledge
all these strategies are restricted to the case of linear operators
or relations and we emphasize here that our approach also works for
nonlinear realizations.

The article is structured as follows. In Section 2 we recall some
well-known facts on maximal monotone relations and we refer the reader
to \cite{Brezis,hu2000handbook,Morosanu} for a detailed study of
this topic. Moreover, we recall the definition and basic properties
of so-called boundary data spaces (see \cite{Picard2012_comprehensive_control}).
In the third section we prove our main theorem (\prettyref{thm:char_max_mon}):
a characterization of all maximal monotone restrictions of operators
of the form \prettyref{eq:block} in terms of the associated boundary
data spaces. One part of this statement was already proved by the
author in \cite{Trostorff2012_nonlin_bd} but for sake of completeness
we state the proof once again. Moreover we give a characterization
of all skew-selfadjoint restrictions (\prettyref{cor:skew-selfadjoint}).
Section 4 is devoted to the comparison of abstract boundary data spaces
as they were introduced in Section 2 and the classical trace spaces
$H^{\frac{1}{2}}(\partial\Omega)$ and $H^{-\frac{1}{2}}(\partial\Omega)$
for bounded Lipschitz-domains $\Omega\subseteq\mathbb{R}^{n}$ (see
e.g. \cite{necas2011direct}). In particular, we show how classical
boundary conditions of Dirichlet-, Neumann- or Robin-type can be formulated
within the framework of boundary data spaces associated with the operators
$\grad$ and $\dive.$ We conclude the article with two applications.
In the first one we study so-called linear Port-Hamiltonian systems
as they were introduced in \cite{Jacob_Zwart2012,leGorrec2005} and
show, how these systems can be embedded into our abstract setting.
In the second example we consider frictional boundary conditions of
monotone type arising in the theory of contact problems in visco-elasticity
(\cite{Migorski_2009,Migorski2011,Sofonea2009}). Moreover, by means
of this example we illustrate, how to formulate boundary conditions
on different parts of the boundary within our framework.

Throughout, all Hilbert spaces are assumed to be complex and the inner
products are denoted by $\langle\cdot|\cdot\rangle,$ which are assumed
to be linear in the second and conjugate linear in the first argument.
The induced norm is denoted by $|\cdot|.$ Moreover, for a Hilbert
space $H$ and a closed subspace $V\subseteq H$ we denote by $\pi_{V}:H\to V$
the orthogonal projection onto $V$. The adjoint $\pi_{V}^{\ast}:V\to H$
is then the canonical embedding and the projector on $V$ is given
by $P_{V}\coloneqq\pi_{V}^{\ast}\pi_{V}:H\to H.$ The operator $\pi_{V}\pi_{V}^{\ast}:V\to V$
is just the identity on $V$.

\section{The framework}

\subsection{Maximal monotone relations}

In this subsection we recall some well-known facts about maximal monotone
relations.%
\footnote{In the literature, the notion of maximal monotonicity is frequently
defined for set valued mappings, i.e. mappings of the form $A:\mathcal{D}(A)\subseteq H\to\mathcal{P}(H).$
However, we prefer the notion of binary relations $A\subseteq H\oplus H$
instead of set-valued mappings, that is, we identify a set-valued
mapping $A$ with the relation 
\[
\left\{ (u,v)\in H\oplus H\,|\, u\in\mathcal{D}(A),\, v\in A(u)\right\} .
\]
} We refer the reader to the monographs \cite{Brezis,hu2000handbook,Morosanu}
and the references therein for a deeper study of maximal monotone
relations and for the proofs of most of the statements in this subsection.
We introduce some algebraic notions for binary relations in order
to work with binary relations in a comfortable way.
\begin{defn*}
Let $H_{0},H_{1}$ be two Hilbert spaces and $A\subseteq H_{0}\oplus H_{1}.$
For $M\subseteq H_{0}$ we define the \emph{post-set of $M$ under
$A$ }by 
\[
A[M]\coloneqq\left\{ y\in H_{1}\,|\,\exists x\in M:(x,y)\in A\right\} .
\]
Analogously, for $N\subseteq H_{1}$ we define the \emph{pre-set of
$N$ under $A$} by 
\[
[N]A\coloneqq\left\{ x\in H_{0}\,|\,\exists y\in N:(x,y)\in A\right\} .
\]
The \emph{inverse relation} $A^{-1}\subseteq H_{1}\oplus H_{0}$ is
given by 
\[
A^{-1}\coloneqq\left\{ (y,x)\in H_{1}\oplus H_{0}\,|\,(x,y)\in A\right\} .
\]
Moreover, for $B\subseteq H_{0}\oplus H_{1}$ and $\lambda\in\mathbb{C}$
we define 
\[
\lambda A+B\coloneqq\left\{ (x,\lambda y+z)\in H_{0}\oplus H_{1}\,|\,(x,y)\in A\wedge(x,z)\in B\right\} .
\]
The relation $A$ is called \emph{bounded, }if for each bounded set
$M\subseteq H_{0}$ the post-set $A[M]$ is also bounded.\\
Finally we define the \emph{adjoint relation }$A^{\ast}$ of $A$
by 
\[
A^{\ast}=\left\{ (-y,x)\in H_{1}\oplus H_{0}\,|\,(x,y)\in A\right\} ^{\bot_{H_{1}\oplus H_{0}}}\subseteq H_{1}\oplus H_{0}.
\]
\end{defn*}
\begin{rem}
\label{rem:adjoint}We note that $A^{\ast}$ is always a closed linear
relation. Moreover, a pair $(u,v)\in H_{1}\oplus H_{0}$ belongs to
$A^{\ast}$ if and only if 
\[
\langle y|u\rangle_{H_{1}}=\langle x|v\rangle_{H_{0}}
\]
for all pairs $(x,y)\in A$, see e.g. \cite[p.14]{Picard_McGhee}.
\end{rem}
We now give the definition of monotonicity and maximal monotonicity
of binary relations.
\begin{defn*}
Let $A\subseteq H\oplus H$. Then $A$ is called \emph{monotone,}
if for each $(u,v),(x,y)\in A$ the inequality 
\[
\Re\langle u-x|v-y\rangle\geq0
\]
holds. A monotone relation $A$ is called \emph{maximal monotone,
}if there exists no proper monotone extension, i.e. for each monotone
$B\subseteq H\oplus H$ with $A\subseteq B$ it follows that $A=B.$ \end{defn*}
\begin{rem}
\label{rem:demi}A maximal monotone relation $A\subseteq H\oplus H$
is \emph{demi-closed}, i.e. for each sequence $\left(\left(x_{n},y_{n}\right)\right)_{n\in\mathbb{N}}$
in $A$, where $x_{n}\rightharpoonup x$ and $y_{n}\to y$ or $x_{n}\to x$
and $y_{n}\rightharpoonup y$ for some $x,y\in H$ as $n\to\infty$
we have $(x,y)\in A$ (see e.g. \cite[Proposition 1.1]{Morosanu}).
Moreover, for each $x\in H$ the post-set $A[\{x\}]$ is closed and
convex.
\end{rem}
Classical examples of maximal monotone relations are skew-selfadjoint
operators, non-negative selfadjoint operators and subgradients of
lower semicontinuous, convex functions (see \cite{Rockafellar}).
In 1962 G.Minty proves the following celebrated characterization for
maximal monotonicity.
\begin{thm}[Minty's Theorem, \cite{Minty}]
\label{thm:Minty} Let $A\subseteq H\oplus H$ be monotone. Then
the following statements are equivalent

\begin{enumerate}[(i)]

\item $A$ is maximal monotone,

\item For all $\lambda>0$ the relation%
\footnote{Here, we denote by $1$ the identity on $H$.%
} $1+\lambda A$ is onto, i.e. $(1+\lambda A)[H]=H,$

\item There exists $\lambda>0$ such that $1+\lambda A$ is onto.

\end{enumerate}
\end{thm}
Using this theorem, we can define the Yosida-approximation of maximal
monotone relations.
\begin{defn*}
Let $A\subseteq H\oplus H$ be maximal monotone and $\lambda>0.$
Then we define the mapping $A_{\lambda}:H\to H$ by 
\[
A_{\lambda}(x)=\lambda^{-1}\left(x-(1+\lambda A)^{-1}(x)\right),
\]
the so-called \emph{Yosida-approximation of $A$. }Note that due to
the monotonicity of $A$, the relation $(1+\lambda A)^{-1}$ defines
a Lipschitz-continuous mapping with smallest Lipschitz-constant less
than or equal to $1$ and by \prettyref{thm:Minty} this mapping is
defined on the whole Hilbert space $H$. Consequently, $A_{\lambda}$
is also a Lipschitz-continuous mapping defined on the whole space
$H$.\end{defn*}
\begin{prop}[{see e.g. \cite[Theorem 1.3]{Morosanu}}]
\label{prop:properties_A} Let $A\subseteq H\oplus H$ be maximal
monotone and set%
\footnote{By $P_{A[\{x\}]}$ we denote the orthogonal projection on the closed
convex set $A[\{x\}].$%
} 
\[
A^{0}(x)\coloneqq P_{A[\{x\}]}(0)\quad(x\in[H]A)
\]
the \emph{principal section of $A$}. Then:

\begin{enumerate}[(a)]

\item For each $\lambda>0$ the mapping $A_{\lambda}$ is maximal
monotone,

\item For each $\lambda>0$ and $x\in H$ we have $\left((1+\lambda A)^{-1}(x),A_{\lambda}(x)\right)\in A$,

\item For each $\lambda>0$ and $x\in[H]A$ we have $|A_{\lambda}(x)|\leq|A^{0}(x)|.$

\end{enumerate}
\end{prop}
We conclude this subsection with two statements about the construction
of maximal monotone relations from given ones.
\begin{prop}
\label{prop:dircet_sum_rel}Let $H_{0},H_{1}$ be two Hilbert spaces
and $B_{0}\subseteq H_{0}\oplus H_{0},\, B_{1}\subseteq H_{1}\oplus H_{1}$
maximal monotone relations. Then 
\[
B_{0}\oplus B_{1}\coloneqq\left\{ \left.\left((u,x),(v,y)\right)\in\left(H_{0}\oplus H_{1}\right)^{2}\,\right|\,(u,v)\in B_{0},(x,y)\in B_{1}\right\} 
\]
defines a maximal monotone relation on $H_{0}\oplus H_{1}.$ \end{prop}
\begin{proof}
The proof is straightforward and we therefore omit it.\end{proof}
\begin{prop}
\label{prop:sim_bd_rel}Let $H_{0},H_{1}$ be two Hilbert spaces and
$B\subseteq H_{0}\oplus H_{0}$ be maximal monotone and bounded. Moreover,
let $T:H_{1}\to H_{0}$ be linear and bounded. If $T[H_{1}]\cap[H_{0}]B\ne\emptyset$,
then 
\[
T^{\ast}BT\coloneqq\left\{ (x,T^{\ast}y)\in H_{1}\oplus H_{1}\,|\,(Tx,y)\in B\right\} 
\]
is maximal monotone.\end{prop}
\begin{proof}
The monotonicity of $T^{\ast}BT$ is obvious. For showing the maximal
monotonicity we use Minty's Theorem (\prettyref{thm:Minty}). For
that purpose, let $f\in H_{1}$. Since for each $\lambda>0$ the mapping
$T^{\ast}B_{\lambda}T:H_{1}\to H_{1}$ is monotone and Lipschitz-continuous
(and hence, maximal monotone cf. \cite[Corollary 2.8]{Trostorff2013_nonautoincl}),
we find $x_{\lambda}\in H_{1}$ such that 
\[
x_{\lambda}+T^{\ast}B_{\lambda}(Tx_{\lambda})=f.
\]
We show that the family $(x_{\lambda})_{\lambda>0}$ is bounded. For
doing so let $x^{\ast}\in H_{1}$ such that $Tx^{\ast}\in[H_{0}]B$.
Then 
\begin{align*}
\Re\langle x_{\lambda}-x^{\ast}|f-(x^{\ast}+T^{\ast}B_{\lambda}(Tx^{\ast}))\rangle & =|x_{\lambda}-x^{\ast}|^{2}+\Re\langle x_{\lambda}-x^{\ast}|T^{\ast}B_{\lambda}(Tx_{\lambda})-T^{\ast}B_{\lambda}(Tx^{\ast})\rangle\\
 & \geq|x_{\lambda}-x^{\ast}|^{2},
\end{align*}
due to the monotonicity of $T^{\ast}B_{\lambda}T.$ The latter implies
\begin{align*}
|x_{\lambda}| & \leq|x_{\lambda}-x^{\ast}|+|x^{\ast}|\\
 & \leq|f|+2|x^{\ast}|+\|T^{\ast}\||B_{\lambda}(Tx^{\ast})|\\
 & \leq|f|+2|x^{\ast}|+\|T^{\ast}\||B^{0}(Tx^{\ast})|
\end{align*}
for all $\lambda>0$, where we have used \prettyref{prop:properties_A}
(c). From the boundedness of $(x_{\lambda})_{\lambda>0}$ we derive
the boundedness of $\left((1+\lambda B)^{-1}\left(Tx_{\lambda}\right)\right)_{\lambda\in]0,1]}.$
Indeed, we estimate for all $\lambda\in]0,1]$ 
\begin{align*}
\left|(1+\lambda B)^{-1}\left(Tx_{\lambda}\right)\right| & \leq\left|(1+\lambda B)^{-1}\left(Tx_{\lambda}\right)-(1+\lambda B)^{-1}\left(Tx^{\ast}\right)\right|+|(1+\lambda B)^{-1}\left(Tx^{\ast}\right)|\\
 & \leq\|T\|(|x_{\lambda}|+|x^{\ast}|)+|\lambda B_{\lambda}\left(Tx^{\ast}\right)+Tx^{\ast}|\\
 & \leq\|T\|(|x_{\lambda}|+2|x^{\ast}|)+|B^{0}\left(Tx^{\ast}\right)|.
\end{align*}
Since $\left((1+\lambda B)^{-1}\left(Tx_{\lambda}\right),B_{\lambda}(Tx_{\lambda})\right)\in B$
(see \prettyref{prop:properties_A} (b)) and since $B$ is bounded,
we obtain 
\[
C\coloneqq\sup_{\lambda\in]0,1]}|B_{\lambda}(Tx_{\lambda})|<\infty.
\]
The latter gives that $(B_{\lambda}(Tx_{\lambda}))_{\lambda>0}$ has
a weak convergent subsequence $(B_{\lambda_{n}}(Tx_{\lambda_{n}}))_{n\in\mathbb{N}}$
with $\lambda_{n}\to0$ and we denote its weak limit by $y$. Let
now $\lambda,\mu\in]0,1].$ Then we compute 
\begin{align*}
 & |x_{\lambda}-x_{\mu}|^{2}\\
 & =\Re\langle x_{\lambda}-x_{\mu}|f-T^{\ast}B_{\lambda}(Tx_{\lambda})-(f-T^{\ast}B_{\mu}(Tx_{\mu}))\rangle\\
 & =\Re\langle Tx_{\lambda}-Tx_{\mu}|B_{\mu}(Tx_{\mu})-B_{\lambda}(Tx_{\lambda})\rangle\\
 & =\Re\langle\lambda B_{\lambda}(Tx_{\lambda})+(1-\lambda B)^{-1}(Tx_{\lambda})-\mu B_{\mu}(Tx_{\mu})-(1-\mu B)^{-1}(Tx_{\mu})|B_{\mu}(Tx_{\mu})-B_{\lambda}(Tx_{\lambda})\rangle\\
 & \leq\Re\langle\lambda B_{\lambda}(Tx_{\lambda})-\mu B_{\mu}(Tx_{\mu})|B_{\mu}(Tx_{\mu})-B_{\lambda}(Tx_{\lambda})\rangle\\
 & \leq2C^{2}(\lambda+\mu),
\end{align*}
where we have again used \prettyref{prop:properties_A} (b). Thus
$\left(x_{\lambda_{n}}\right)_{n\in\mathbb{N}}$ is a Cauchy-sequence
and hence, it converges and we denote its limit by $x$. By the continuity
of $T$ we have $Tx_{\lambda_{n}}\to Tx$ as $n\to\infty$ and
\[
\left|(1+\lambda_{n}B)^{-1}\left(Tx_{\lambda_{n}}\right)-Tx\right|\leq\lambda_{n}|B_{\lambda_{n}}(Tx_{\lambda_{n}})|+|Tx_{\lambda_{n}}-Tx|\to0\quad(n\to\infty).
\]
By the demi-closedness of $B$ (see \prettyref{rem:demi}) we get
that $(Tx,y)\in B,$ which implies $(x,T^{\ast}y)\in T^{\ast}BT.$
Moreover 
\begin{align*}
f & =\wlim_{n\to\infty}\left(x_{\lambda_{n}}+T^{\ast}B_{\lambda_{n}}(Tx_{\lambda_{n}})\right)=x+T^{\ast}y,
\end{align*}
or in other words $(x,f)\in1+T^{\ast}BT.$ \end{proof}
\begin{rem}
A similar result was shown by Robinson \cite{Robinson1999} without
imposing boundedness of $B$, but with an additional compatibility
assumption on $T$ and $B$ and assuming the closedness of the ranges
of $T$ and $T^{\ast}.$ 
\end{rem}

\subsection{Boundary data spaces}

In this subsection we recall the notion and some basic properties
of boundary data spaces as they were introduced in \cite{Picard2012_comprehensive_control}.
Throughout, let $H_{0},H_{1}$ be two Hilbert spaces and $G_{c}:\mathcal{D}(G_{c})\subseteq H_{0}\to H_{1},\: D_{c}:\mathcal{D}(D_{c})\subseteq H_{1}\to H_{0}$
be two densely defined, closed linear operators with $G_{c}\subseteq-D_{c}^{\ast}$
and consequently $D_{c}\subseteq-G_{c}^{\ast}.$ We set $G\coloneqq-D_{c}^{\ast}$
and $D\coloneqq-G_{c}^{\ast}.$ 
\begin{example}
As a guiding example for the situation above we set $H_{0}\coloneqq L_{2}(\Omega)$
and $H_{1}\coloneqq L_{2}(\Omega)^{n}$ for some open $\Omega\subseteq\mathbb{R}^{n}.$
We define the gradient $\grad_{c}$ with ``vanishing trace'' as
the closure of 
\begin{align*}
\grad_{c}|_{C_{c}^{\infty}(\Omega)}:C_{c}^{\infty}(\Omega)\subseteq L_{2}(\Omega) & \to L_{2}(\Omega)^{n}\\
\phi & \mapsto(\partial_{i}\phi)_{i\in\{1,\ldots,n\}},
\end{align*}
where we denote by $C_{c}^{\infty}(\Omega)$ the space of arbitrarily
differentiable functions with compact support in $\Omega.$ The domain
of $\grad_{c}$ then coincides with the classical Sobolev space $W_{2,0}^{1}(\Omega).$
Analogously, we define the operator $\dive_{c}$ as the closure of
\begin{align*}
\dive_{c}|_{C_{c}^{\infty}(\Omega)^{n}}:C_{c}^{\infty}(\Omega)^{n}\subseteq L_{2}(\Omega)^{n} & \to L_{2}(\Omega)\\
(\psi_{i})_{i\in\{1,\ldots,n\}} & \mapsto\sum_{i=1}^{n}\partial_{i}\psi_{i}.
\end{align*}
The domain of $\dive_{c}$ then consists of those $L_{2}$-vector
fields whose distributional divergence is an $L_{2}$-function and
which satisfy an abstract Neumann-boundary condition.%
\footnote{Indeed, if $\partial\Omega$ is regular enough then $\psi\in\mathcal{D}(\dive_{c})$
satisfies $N\cdot\psi=0$, where $N$ denotes the unit outward normal
vector field at $\partial\Omega$, see Section 4.%
} We set $\grad\coloneqq-\dive_{c}^{\ast}$ and $\dive\coloneqq-\grad_{c}^{\ast}$
and get $\grad_{c}\subseteq\grad$ as well as $\dive_{c}\subseteq\dive$.
The domains of $\grad$ and $\dive$ are then the maximal sets of
$L_{2}$-functions or -vector fields such that the distributional
gradient or divergence is again an $L_{2}$-vector field or -function,
respectively.
\end{example}
We recall the notion of short Sobolev-chains (see \cite[Section 2.1]{Picard_McGhee}).
\begin{defn*}
Let $C:\mathcal{D}(C)\subseteq H\to H$ be a closed, densely defined
linear operator with $0\in\rho(C).$ Then we denote by $H^{1}(C)$
the Hilbert space given by the domain $\mathcal{D}(C)$ equipped with
the inner product $\langle\cdot|\cdot\rangle_{H^{1}(C)}\coloneqq\langle C\cdot|C\cdot\rangle_{H}.$
Moreover, we set $H^{-1}(C)$ as the completion of $H$ with respect
to the norm induced by the inner product $\langle\cdot|\cdot\rangle_{H^{-1}(C)}\coloneqq\langle C^{-1}\cdot|C^{-1}\cdot\rangle_{H}.$
Then 
\[
H^{1}(C)\hookrightarrow H\hookrightarrow H^{-1}(C)
\]
 with continuous and dense embeddings and we call the triple $(H^{1}(C),H,H^{-1}(C))$
the \emph{short Sobolev-chain associated with $C$.}\end{defn*}
\begin{rem}
It is easy to see that the operator $C:H^{1}(C)\to H$ is unitary.
Moreover, the operator $C:\mathcal{D}(C)\subseteq H\to H^{-1}(C)$
has a unitary extension which we also denote by $C$. \end{rem}
\begin{prop}[{\cite[Lemma 2.1.16]{Picard_McGhee}}]
\label{prop:modulus}Let $G:\mathcal{D}(G)\subseteq H_{0}\to H_{1}$
be a closed densely defined linear operator. Then%
\footnote{Recall that $|G|\coloneqq\sqrt{G^{\ast}G}$ is a self-adjoint operator
and thus, $0\in\rho(|G|+\i)$. %
} $G:H^{1}(|G|+\i)\to H_{1}$ is bounded. Moreover, the operator $G:\mathcal{D}(G)\subseteq H_{0}\to H^{-1}(|G^{\ast}|+\i)$
has a unique bounded extension.\end{prop}
\begin{defn*}[{\cite[Section 5.2]{Picard2012_comprehensive_control}}]
Let $G_{c},D_{c},G$ and $D$ as above. We define 
\begin{align*}
\mathcal{BD}(G) & \coloneqq\mathcal{D}(G_{c})^{\bot_{H^{1}(|G|+\i)}}=\mathcal{N}(1-DG),\\
\mathcal{BD}(D) & \coloneqq\mathcal{D}(D_{c})^{\bot_{H^{1}(|D|+\i)}}=\mathcal{N}(1-GD),
\end{align*}
the so-called \emph{boundary data spaces, }where the orthogonal complements
are taken with respect to the inner products in $H^{1}(|G|+\i)$ and
$H^{1}(|D|+\i),$ respectively.\end{defn*}
\begin{rem}
According to the projection theorem we have 
\begin{align*}
H^{1}(|G|+\i) & =H^{1}(|G_{c}|+\i)\oplus\mathcal{BD}(G),\\
H^{1}(|D|+\i) & =H^{1}(|D_{c}|+\i)\oplus\mathcal{BD}(D).
\end{align*}
This could be interpreted as a decomposition result for elements in
$H^{1}(|G|+\i)$ and $H^{1}(|D|+\i)$ into one part with ``vanishing
trace'' (in $H^{1}(|G_{c}|+\i)$ or $H^{1}(|D_{c}|+\i)$, respectively)
and one part carrying the whole information about the behaviour at
the boundary.
\end{rem}
Finally, we recall the following result from \cite{Picard2012_comprehensive_control}.
\begin{prop}[{\cite[Theorem 5.2]{Picard2012_comprehensive_control}}]
 Let $G_{c},G,D_{c},D$ as above. Then $G[\mathcal{BD}(G)]\subseteq\mathcal{BD}(D)$
and $D[\mathcal{BD}(D)]\subseteq\mathcal{BD}(G).$ Moreover, the operators
\[
\stackrel{\bullet}{G}:\mathcal{BD}(G)\to\mathcal{BD}(D)
\]
and 
\[
\stackrel{\bullet}{D}:\mathcal{BD}(D)\to\mathcal{BD}(G),
\]
defined as the restrictions%
\footnote{In other words we have 
\[
\stackrel{\bullet}{G}=\pi_{\mathcal{BD}(D)}G\pi_{\mathcal{BD}(G)}^{\ast}
\]
and analogously 
\[
\stackrel{\bullet}{D}=\pi_{\mathcal{BD}(G)}D\pi_{\mathcal{BD}(D)}^{\ast}.
\]
} of $G$ and $D$, respectively, are unitary with $\left(\stackrel{\bullet}{G}\right)^{\ast}=\stackrel{\bullet}{D}.$
\end{prop}

\section{A characterization of maximal monotone realizations}

In this section we give a characterization for maximal monotone realizations
for a certain class of operators in terms of the corresponding boundary
data spaces. As in Subsection 2.2, let $H_{0},H_{1}$ be two Hilbert
spaces and $G_{c}:\mathcal{D}(G_{c})\subseteq H_{0}\to H_{1},\: D_{c}:\mathcal{D}(D_{c})\subseteq H_{1}\to H_{0}$
be two densely defined, closed linear operators with $D_{c}\subseteq-G_{c}^{\ast}$.
We set $D\coloneqq-G_{c}^{\ast}$ and $G\coloneqq-D_{c}^{\ast}$,
which in particular yields $D_{c}\subseteq D$ as well as $G_{c}\subseteq G$.
Let $A:\mathcal{D}(A)\subseteq H_{0}\oplus H_{1}\to H_{0}\oplus H_{1}$
a possibly nonlinear operator with 
\[
A\subseteq\left(\begin{array}{cc}
0 & D\\
G & 0
\end{array}\right).
\]
Recall that we have the following orthogonal decompositions 
\begin{align*}
H^{1}(|G|+\i) & =H^{1}(|G_{c}|+\i)\oplus\mathcal{BD}(G),\\
H^{1}(|D|+\i) & =H^{1}(|D_{c}|+\i)\oplus\mathcal{BD}(D).
\end{align*}
The corresponding projections will be denoted by 
\begin{align*}
\pi_{G_{c}}:H^{1}(|G|+\i) & \to H^{1}(|G_{c}|+\i),\\
\pi_{\mathcal{BD}(G)}:H^{1}(|G|+\i) & \to\mathcal{BD}(G)
\end{align*}
and 
\begin{align*}
\pi_{D_{c}}:H^{1}(|D|+\i) & \to H^{1}(|D_{c}|+\i),\\
\pi_{\mathcal{BD}(D)}:H^{1}(|D|+\i) & \to\mathcal{BD}(D).
\end{align*}
Our main theorem reads as follows.
\begin{thm}
\label{thm:char_max_mon}The operator $A$ is maximal monotone if
and only if there exists a maximal monotone relation $h\subseteq\mathcal{BD}(G)\oplus\mathcal{BD}(G)$
such that 
\begin{equation}
\mathcal{D}(A)=\left\{ (u,v)\in\mathcal{D}(G)\times\mathcal{D}(D)\,\left|\,\left(\pi_{\mathcal{BD}(G)}u,\stackrel{\bullet}{D}\pi_{\mathcal{BD}(D)}v\right)\in h\right.\right\} .\label{eq:D(A)}
\end{equation}

\end{thm}
In \cite{Trostorff2012_nonlin_bd} it was proved that $A$ is maximal
monotone if $\mathcal{D}(A)$ is given by \prettyref{eq:D(A)} for
a maximal monotone relation $h$. However, for sake of completeness
we will recall this result below. First we start with the following
observation.
\begin{lem}
\label{lem:inner_prod}Let $(u,v)\in\mathcal{D}(G)\times\mathcal{D}(D).$
Then 
\[
\Re\left\langle \left.\left(\begin{array}{cc}
0 & D\\
G & 0
\end{array}\right)\left(\begin{array}{c}
u\\
v
\end{array}\right)\right|\left(\begin{array}{c}
u\\
v
\end{array}\right)\right\rangle _{H_{0}\oplus H_{1}}=\Re\langle\pi_{\mathcal{BD}(G)}u|\stackrel{\bullet}{D}\pi_{\mathcal{BD}(D)}v\rangle_{\mathcal{BD}(G)}.
\]
\end{lem}
\begin{proof}
We compute 
\begin{align*}
 & \Re\left\langle \left.\left(\begin{array}{cc}
0 & D\\
G & 0
\end{array}\right)\left(\begin{array}{c}
u\\
v
\end{array}\right)\right|\left(\begin{array}{c}
u\\
v
\end{array}\right)\right\rangle _{H_{0}\oplus H_{1}}\\
 & =\Re\langle Dv|u\rangle_{H_{0}}+\Re\langle Gu|v\rangle_{H_{1}}\\
 & =\Re\langle D_{c}\pi_{D_{c}}^{\ast}\pi_{D_{c}}v|u\rangle_{H_{0}}+\Re\langle D\pi_{\mathcal{BD}(D)}^{\ast}\pi_{\mathcal{BD}(D)}v|u\rangle_{H_{0}}\\
 & \quad+\Re\langle Gu|\pi_{D_{c}}^{\ast}\pi_{D_{c}}v\rangle_{H_{1}}+\Re\langle Gu|\pi_{\mathcal{BD}(D)}^{\ast}\pi_{\mathcal{BD}(D)}v\rangle_{H_{1}}\\
 & =\Re\langle D\pi_{\mathcal{BD}(D)}^{\ast}\pi_{\mathcal{BD}(D)}v|u\rangle_{H_{0}}+\Re\langle Gu|\pi_{\mathcal{BD}(D)}^{\ast}\pi_{\mathcal{BD}(D)}v\rangle_{H_{1}}\\
 & =\Re\langle D\pi_{\mathcal{BD}(D)}^{\ast}\pi_{\mathcal{BD}(D)}v|\pi_{G_{c}}^{\ast}\pi_{G_{c}}u\rangle_{H_{0}}+\Re\langle D\pi_{\mathcal{BD}(D)}^{\ast}\pi_{\mathcal{BD}(D)}v|\pi_{\mathcal{BD}(G)}^{\ast}\pi_{\mathcal{BD}(G)}u\rangle_{H_{0}}\\
 & \quad+\Re\langle G_{c}\pi_{G_{c}}^{\ast}\pi_{G_{c}}u|\pi_{\mathcal{BD}(D)}^{\ast}\pi_{\mathcal{BD}(D)}v\rangle_{H_{1}}+\Re\langle G\pi_{\mathcal{BD}(G)}^{\ast}\pi_{\mathcal{BD}(G)}u|\pi_{\mathcal{BD}(D)}^{\ast}\pi_{\mathcal{BD}(D)}v\rangle_{H_{1}}\\
 & =\Re\langle D\pi_{\mathcal{BD}(D)}^{\ast}\pi_{\mathcal{BD}(D)}v|\pi_{\mathcal{BD}(G)}^{\ast}\pi_{\mathcal{BD}(G)}u\rangle_{H_{0}}+\Re\langle G\pi_{\mathcal{BD}(G)}^{\ast}\pi_{\mathcal{BD}(G)}u|\pi_{\mathcal{BD}(D)}^{\ast}\pi_{\mathcal{BD}(D)}v\rangle_{H_{1}}\\
 & =\Re\langle D\pi_{\mathcal{BD}(D)}^{\ast}\pi_{\mathcal{BD}(D)}v|\pi_{\mathcal{BD}(G)}^{\ast}\pi_{\mathcal{BD}(G)}u\rangle_{H_{0}}+\Re\langle G\pi_{\mathcal{BD}(G)}^{\ast}\pi_{\mathcal{BD}(G)}u|GD\pi_{\mathcal{BD}(D)}^{\ast}\pi_{\mathcal{BD}(D)}v\rangle_{H_{1}}\\
 & =\Re\langle\pi_{\mathcal{BD}(G)}^{\ast}\pi_{\mathcal{BD}(G)}u|D\pi_{\mathcal{BD}(D)}^{\ast}\pi_{\mathcal{BD}(D)}v\rangle_{H^{1}(|G|+\i)}\\
 & =\Re\langle\pi_{\mathcal{BD}(G)}u|\stackrel{\bullet}{D}\pi_{\mathcal{BD}(D)}v\rangle_{\mathcal{BD}(G)}.\tag*{\qedhere}
\end{align*}
\end{proof}
\begin{lem}
\label{lem:domain_proj}Let $A$ be maximal monotone and $(u,v)\in\mathcal{D}(G)\times\mathcal{D}(D)$.
Then $(u,v)\in\mathcal{D}(A)$ if and only if $(\pi_{\mathcal{BD}(G)}^{\ast}\pi_{\mathcal{BD}(G)}u,\pi_{\mathcal{BD}(D)}^{\ast}\pi_{\mathcal{BD}(D)}v)\in\mathcal{D}(A).$\end{lem}
\begin{proof}
We compute for every $(x,y)\in\mathcal{D}(A)$ using \prettyref{lem:inner_prod}
\begin{align*}
 & \Re\left\langle \left.\left(\begin{array}{cc}
0 & D\\
G & 0
\end{array}\right)\left(\begin{array}{c}
\pi_{\mathcal{BD}(G)}^{\ast}\pi_{\mathcal{BD}(G)}u\\
\pi_{\mathcal{BD}(D)}^{\ast}\pi_{\mathcal{BD}(D)}v
\end{array}\right)-A\left(\begin{array}{c}
x\\
y
\end{array}\right)\right|\left(\begin{array}{c}
\pi_{\mathcal{BD}(G)}^{\ast}\pi_{\mathcal{BD}(G)}u\\
\pi_{\mathcal{BD}(D)}^{\ast}\pi_{\mathcal{BD}(D)}v
\end{array}\right)-\left(\begin{array}{c}
x\\
y
\end{array}\right)\right\rangle _{H_{0}\oplus H_{1}}\\
 & =\Re\left\langle \left.\left(\begin{array}{cc}
0 & D\\
G & 0
\end{array}\right)\left(\begin{array}{c}
\pi_{\mathcal{BD}(G)}^{\ast}\pi_{\mathcal{BD}(G)}u-x\\
\pi_{\mathcal{BD}(D)}^{\ast}\pi_{\mathcal{BD}(D)}v-y
\end{array}\right)\right|\left(\begin{array}{c}
\pi_{\mathcal{BD}(G)}^{\ast}\pi_{\mathcal{BD}(G)}u-x\\
\pi_{\mathcal{BD}(D)}^{\ast}\pi_{\mathcal{BD}(D)}v-y
\end{array}\right)\right\rangle _{H_{0}\oplus H_{1}}\\
 & =\Re\langle\pi_{\mathcal{BD}(G)}(\pi_{\mathcal{BD}(G)}^{\ast}\pi_{\mathcal{BD}(G)}u-x)|\stackrel{\bullet}{D}\pi_{\mathcal{BD}(D)}(\pi_{\mathcal{BD}(D)}^{\ast}\pi_{\mathcal{BD}(D)}v-y)\rangle_{\mathcal{BD}(G)}\\
 & =\Re\langle\pi_{\mathcal{BD}(G)}(u-x)|\stackrel{\bullet}{D}\pi_{\mathcal{BD}(D)}(v-y)\rangle_{\mathcal{BD}(G)}\\
 & =\Re\left\langle \left.\left(\begin{array}{cc}
0 & D\\
G & 0
\end{array}\right)\left(\begin{array}{c}
u-x\\
v-y
\end{array}\right)\right|\left(\begin{array}{c}
u-x\\
v-y
\end{array}\right)\right\rangle _{H_{0}\oplus H_{1}}\\
 & =\Re\left\langle \left.\left(\begin{array}{cc}
0 & D\\
G & 0
\end{array}\right)\left(\begin{array}{c}
u\\
v
\end{array}\right)-A\left(\begin{array}{c}
x\\
y
\end{array}\right)\right|\left(\begin{array}{c}
u\\
v
\end{array}\right)-\left(\begin{array}{c}
x\\
y
\end{array}\right)\right\rangle _{H_{0}\oplus H_{1}}.
\end{align*}
If $(u,v)\in\mathcal{D}(A)$ the last term is non-negative and thus
\[
(\pi_{\mathcal{BD}(G)}^{\ast}\pi_{\mathcal{BD}(G)}u,\pi_{\mathcal{BD}(D)}^{\ast}\pi_{\mathcal{BD}(D)}v)\in\mathcal{D}(A),
\]
according to the maximality of $A$. If on the other hand 
\[
(\pi_{\mathcal{BD}(G)}^{\ast}\pi_{\mathcal{BD}(G)}u,\pi_{\mathcal{BD}(D)}^{\ast}\pi_{\mathcal{BD}(D)}v)\in\mathcal{D}(A),
\]
then the first term in the latter equalities is non-negative and hence,
again by the maximality of $A$ we deduce $(u,v)\in\mathcal{D}(A)$.\end{proof}
\begin{prop}
\label{prop:existence_h}Let $A$ be maximal monotone. Then there
exists a relation $h\subseteq\mathcal{BD}(G)\oplus\mathcal{BD}(G)$
such that 
\[
\mathcal{D}(A)=\left\{ (u,v)\in\mathcal{D}(G)\times\mathcal{D}(D)\,\left|\,(\pi_{\mathcal{BD}(G)}u,\stackrel{\bullet}{D}\pi_{\mathcal{BD}(D)}v)\in h\right.\right\} .
\]
\end{prop}
\begin{proof}
We define 
\[
h\coloneqq\left\{ (x,y)\in\mathcal{BD}(G)\oplus\mathcal{BD}(G)\,\left|\,(\pi_{\mathcal{BD}(G)}^{\ast}x,\pi_{\mathcal{BD}(D)}^{\ast}\stackrel{\bullet}{G}y)\in\mathcal{D}(A)\right.\right\} .
\]
Let $(u,v)\in\mathcal{D}(G)\times\mathcal{D}(D).$ Then, using \prettyref{lem:domain_proj},
we get that 
\begin{align*}
(\pi_{\mathcal{BD}(G)}u,\stackrel{\bullet}{D}\pi_{\mathcal{BD}(D)}v)\in h & \Leftrightarrow(\pi_{\mathcal{BD}(G)}^{\ast}\pi_{\mathcal{BD}(G)}u,\pi_{\mathcal{BD}(D)}^{\ast}\stackrel{\bullet}{G}\stackrel{\bullet}{D}\pi_{\mathcal{BD}(D)}v)\in\mathcal{D}(A)\\
 & \Leftrightarrow(\pi_{\mathcal{BD}(G)}^{\ast}\pi_{\mathcal{BD}(G)}u,\pi_{\mathcal{BD}(D)}^{\ast}\pi_{\mathcal{BD}(D)}v)\in\mathcal{D}(A)\\
 & \Leftrightarrow(u,v)\in\mathcal{D}(A).\tag*{\qedhere}
\end{align*}
\end{proof}
\begin{prop}
\label{prop:h_max_mon}Let $A$ be maximal monotone and $h\subseteq\mathcal{BD}(G)\oplus\mathcal{BD}(G)$
such that 
\[
\mathcal{D}(A)=\left\{ (u,v)\in\mathcal{D}(G)\times\mathcal{D}(D)\,\left|\,(\pi_{\mathcal{BD}(G)}u,\stackrel{\bullet}{D}\pi_{\mathcal{BD}(D)}v)\in h\right.\right\} .
\]
Then $h$ is maximal monotone.\end{prop}
\begin{proof}
Let $(x,y),(w,z)\in h.$ Then $(\pi_{\mathcal{BD}(G)}^{\ast}x,\pi_{\mathcal{BD}(D)}^{\ast}\stackrel{\bullet}{G}y),(\pi_{\mathcal{BD}(G)}^{\ast}w,\pi_{\mathcal{BD}(D)}^{\ast}\stackrel{\bullet}{G}z)\in\mathcal{D}(A)$.
By \prettyref{lem:inner_prod} we obtain 
\begin{align*}
 & \Re\langle x-w|y-z\rangle_{\mathcal{BD}(G)}\\
 & =\Re\left\langle \pi_{\mathcal{BD}(G)}\left(\pi_{\mathcal{BD}(G)}^{\ast}x-\pi_{\mathcal{BD}(G)}^{\ast}w\right)\left|\stackrel{\bullet}{D}\pi_{\mathcal{BD}(D)}\left(\pi_{\mathcal{BD}(D)}^{\ast}\stackrel{\bullet}{G}y-\pi_{\mathcal{BD}(D)}^{\ast}\stackrel{\bullet}{G}z\right)\right.\right\rangle _{\mathcal{BD}(G)}\\
 & =\Re\left\langle \left.\left(\begin{array}{cc}
0 & D\\
G & 0
\end{array}\right)\left(\begin{array}{c}
\pi_{\mathcal{BD}(G)}^{\ast}x-\pi_{\mathcal{BD}(G)}^{\ast}w\\
\pi_{\mathcal{BD}(D)}^{\ast}\stackrel{\bullet}{G}y-\pi_{\mathcal{BD}(D)}^{\ast}\stackrel{\bullet}{G}z
\end{array}\right)\right|\left(\begin{array}{c}
\pi_{\mathcal{BD}(G)}^{\ast}x-\pi_{\mathcal{BD}(G)}^{\ast}w\\
\pi_{\mathcal{BD}(D)}^{\ast}\stackrel{\bullet}{G}y-\pi_{\mathcal{BD}(D)}^{\ast}\stackrel{\bullet}{G}z
\end{array}\right)\right\rangle _{H_{0}\oplus H_{1}}\\
 & =\Re\left\langle \left.A\left(\begin{array}{c}
\pi_{\mathcal{BD}(G)}^{\ast}x\\
\pi_{\mathcal{BD}(D)}^{\ast}\stackrel{\bullet}{G}y
\end{array}\right)-A\left(\begin{array}{c}
\pi_{\mathcal{BD}(G)}^{\ast}w\\
\pi_{\mathcal{BD}(D)}^{\ast}\stackrel{\bullet}{G}z
\end{array}\right)\right|\left(\begin{array}{c}
\pi_{\mathcal{BD}(G)}^{\ast}x\\
\pi_{\mathcal{BD}(D)}^{\ast}\stackrel{\bullet}{G}y
\end{array}\right)-\left(\begin{array}{c}
\pi_{\mathcal{BD}(G)}^{\ast}w\\
\pi_{\mathcal{BD}(D)}^{\ast}\stackrel{\bullet}{G}z
\end{array}\right)\right\rangle _{H_{0}\oplus H_{1}}\geq0,
\end{align*}
which proves the monotonicity of $h$. For showing the maximal monotonicity
we use Minty's Theorem. Let $f\in\mathcal{BD}(G).$ Then there exists
$(u,v)\in\mathcal{D}(A)$ such that 
\[
\left(\begin{array}{c}
u\\
v
\end{array}\right)+A\left(\begin{array}{c}
u\\
v
\end{array}\right)=\left(\begin{array}{c}
\pi_{\mathcal{BD}(G)}^{\ast}f\\
\pi_{\mathcal{BD}(D)}^{\ast}\stackrel{\bullet}{G}f
\end{array}\right),
\]
i.e. 
\begin{align}
u+Dv & =\pi_{\mathcal{BD}(G)}^{\ast}f,\label{eq:max_mon_h}\\
v+Gu & =\pi_{\mathcal{BD}(D)}^{\ast}\stackrel{\bullet}{G}f.\nonumber 
\end{align}
The latter especially yields that $Dv\in\mathcal{D}(G)$ and $Gu\in\mathcal{D}(D).$
Moreover, we get that 
\begin{align*}
v-GDv & =v-G(\pi_{\mathcal{BD}(G)}^{\ast}f-u)\\
 & =v-\left(\pi_{\mathcal{BD}(D)}^{\ast}\stackrel{\bullet}{G}f-Gu\right)\\
 & =v-v\\
 & =0
\end{align*}
yielding that $v=\pi_{\mathcal{BD}(D)}^{\ast}\pi_{\mathcal{BD}(D)}v.$
Thus, the first equality in \prettyref{eq:max_mon_h} gives 
\[
\pi_{\mathcal{BD}(G)}u+\stackrel{\bullet}{D}\pi_{\mathcal{BD}(D)}v=f.
\]
Moreover, by definition $\left(\pi_{\mathcal{BD}(G)}u,\stackrel{\bullet}{D}\pi_{\mathcal{BD}(D)}v\right)\in h,$
which yields 
\[
(\pi_{\mathcal{BD}(G)}u,f)\in1+h.
\]
Hence, we have found out that $(1+h)[\mathcal{BD}(G)]=\mathcal{BD}(G)$,
which implies the maximal monotonicity of $h$. 
\end{proof}
Together with \prettyref{prop:existence_h}, the latter proposition
shows one implication in \prettyref{thm:char_max_mon}. For the missing
implication we recall the result and the proof of \cite[Theorem 4.1]{Trostorff2012_nonlin_bd}. 
\begin{prop}
Let $h\subseteq\mathcal{BD}(G)\oplus\mathcal{BD}(G)$ be maximal monotone
and 
\[
\mathcal{D}(A)=\left\{ (u,v)\in\mathcal{D}(G)\times\mathcal{D}(D)\,\left|\,(\pi_{\mathcal{BD}(G)}u,\stackrel{\bullet}{D}\pi_{\mathcal{BD}(D)}v)\in h\right.\right\} .
\]
Then $A$ is maximal monotone.\end{prop}
\begin{proof}
First we prove that $A$ is monotone. For doing so let $(u,v),(x,y)\in\mathcal{D}(A).$
Then by \prettyref{lem:inner_prod} we obtain 
\begin{align*}
 & \Re\left\langle \left.A\left(\begin{array}{c}
u\\
v
\end{array}\right)-A\left(\begin{array}{c}
x\\
y
\end{array}\right)\right|\left(\begin{array}{c}
u\\
v
\end{array}\right)-\left(\begin{array}{c}
x\\
y
\end{array}\right)\right\rangle _{H_{0}\oplus H_{1}}\\
 & =\Re\left\langle \pi_{\mathcal{BD}(G)}u-\pi_{\mathcal{BD}(G)}x\left|\stackrel{\bullet}{D}\pi_{\mathcal{BD}(D)}v-\stackrel{\bullet}{D}\pi_{\mathcal{BD}(D)}y\right.\right\rangle _{\mathcal{BD}(G)}\geq0,
\end{align*}
which shows the monotonicity of $A$. Next, we prove that $A$ is
closed. For that purpose let $\left((u_{n},v_{n})\right)_{n\in\mathbb{N}}$
be a sequence in $\mathcal{D}(A)$ such that $(u_{n},v_{n})\to(u,v)$
as $n\to\infty$ and $\left(A\left(\begin{array}{c}
u_{n}\\
v_{n}
\end{array}\right)\right)_{n\in\mathbb{N}}$ is convergent. By the closedness of $G$ and $D$ we obtain $(u,v)\in\mathcal{D}(G)\times\mathcal{D}(D)$
and $u_{n}\to u$ as well as $v_{n}\to v$ in $H^{1}(|G|+\i)$ and
$H^{1}(|D|+\i),$ respectively. Thus $\pi_{\mathcal{BD}(G)}u_{n}\to\pi_{\mathcal{BD}(G)}u$
and $\stackrel{\bullet}{D}\pi_{\mathcal{BD}(D)}v_{n}\to\stackrel{\bullet}{D}\pi_{\mathcal{BD}(D)}v.$
The closedness of $h$ now yields the assertion. Finally, we prove
the maximality of $A$ by using Minty's Theorem. We note that since
$A$ is monotone and closed, it suffices to prove that $1+A$ has
dense range. So let $\left(f,g\right)\in\mathcal{D}(G_{c})\times\mathcal{D}(D_{c})$
and define%
\footnote{Recall that $-DG_{c}=G_{c}^{\ast}G_{c}$ and $-G_{c}D=D^{\ast}D$
are non-negative selfadjoint operators and hence, $1-DG_{c}$ and
$1-G_{c}D$ are boundedly invertible.%
} 
\begin{align*}
\tilde{u} & \coloneqq(1-DG_{c})^{-1}\left(f-D_{c}g\right)\in\mathcal{D}(DG_{c})\\
\tilde{v} & \coloneqq(1-G_{c}D)^{-1}(g-G_{c}f)\in\mathcal{D}(G_{c}D).
\end{align*}
Then
\begin{align*}
\tilde{u}+D\tilde{v} & =(1-DG_{c})^{-1}\left(f-D_{c}g\right)+D(1-G_{c}D)^{-1}(g-G_{c}f)\\
 & =(1-DG_{c})^{-1}\left(f-D_{c}g+(1-DG_{c})D(1-G_{c}D)^{-1}(g-G_{c}f)\right)\\
 & =(1-DG_{c})^{-1}\left(f-D_{c}g+D(1-G_{c}D)(1-G_{c}D)^{-1}(g-G_{c}f)\right)\\
 & =(1-DG_{c})^{-1}(f-DG_{c}f)\\
 & =f
\end{align*}
and analogously 
\[
\tilde{v}+G\tilde{u}=(1-G_{c}D)^{-1}(g-G_{c}f)+G(1-DG_{c})^{-1}\left(f-D_{c}g\right)=g.
\]
Moreover, we define 
\begin{align*}
u\coloneqq\tilde{u}+\pi_{\mathcal{BD}(G)}^{\ast}(1+h)^{-1}\left(-\stackrel{\bullet}{D}\pi_{\mathcal{BD}(D)}G_{c}\tilde{u}\right) & \in\mathcal{D}(G),\\
v\coloneqq\tilde{v}-\pi_{\mathcal{BD}(D)}^{\ast}\stackrel{\bullet}{G}(1+h)^{-1}\left(-\stackrel{\bullet}{D}\pi_{\mathcal{BD}(D)}G_{c}\tilde{u}\right) & \in\mathcal{D}(D).
\end{align*}
Then clearly 
\begin{align*}
u+Dv & =\tilde{u}+D\tilde{v}=f,\\
v+Gu & =\tilde{v}+G\tilde{u}=g.
\end{align*}
Moreover, since $\tilde{u}\in\mathcal{D}(G_{c})$ we have that
\[
\pi_{\mathcal{BD}(G)}u=(1+h)^{-1}\left(-\stackrel{\bullet}{D}\pi_{\mathcal{BD}(D)}G_{c}\tilde{u}\right).
\]
Using that 
\begin{align*}
\pi_{\mathcal{BD}(D)}G_{c}\tilde{u} & =\pi_{\mathcal{BD}(D)}G_{c}(1-DG_{c})^{-1}\left(f-D_{c}g\right)\\
 & =\pi_{\mathcal{BD}(D)}(1-G_{c}D)^{-1}G_{c}f+\pi_{\mathcal{BD}(D)}g-\pi_{\mathcal{BD}(D)}(1-G_{c}D)^{-1}g\\
 & =\pi_{\mathcal{BD}(D)}(1-G_{c}D)^{-1}\left(G_{c}f-g\right)\\
 & =-\pi_{\mathcal{BD}(D)}\tilde{v}
\end{align*}
we obtain 
\begin{align*}
\stackrel{\bullet}{D}\pi_{\mathcal{BD}(D)}v+\pi_{\mathcal{BD}(G)}u & =\stackrel{\bullet}{D}\pi_{\mathcal{BD}(D)}\tilde{v}-(1+h)^{-1}\left(-\stackrel{\bullet}{D}\pi_{\mathcal{BD}(D)}G_{c}\tilde{u}\right)+(1+h)^{-1}\left(-\stackrel{\bullet}{D}\pi_{\mathcal{BD}(D)}G_{c}\tilde{u}\right)\\
 & =\stackrel{\bullet}{D}\pi_{\mathcal{BD}(D)}\tilde{v}\\
 & =-\stackrel{\bullet}{D}\pi_{\mathcal{BD}(D)}G_{c}\tilde{u}.
\end{align*}
Thus, 
\[
\pi_{\mathcal{BD}(G)}u=(1+h)^{-1}\left(-\stackrel{\bullet}{D}\pi_{\mathcal{BD}(D)}G_{c}\tilde{u}\right)=(1+h)^{-1}\left(\stackrel{\bullet}{D}\pi_{\mathcal{BD}(D)}v+\pi_{\mathcal{BD}(G)}u\right),
\]
which is equivalent to $(\pi_{\mathcal{BD}(G)}u,\stackrel{\bullet}{D}\pi_{\mathcal{BD}(D)}v)\in h,$
which yields $(u,v)\in\mathcal{D}(A).$
\end{proof}
We conclude this section with a characterization of all skew-selfadjoint
realizations of $A$.
\begin{prop}
Let $A$ be linear, where $\mathcal{D}(A)$ is given by \prettyref{eq:D(A)}
for some linear relation $h\subseteq\mathcal{BD}(G)\oplus\mathcal{BD}(G)$.
Then $A$ is densely defined and $A^{\ast}\subseteq-\left(\begin{array}{cc}
0 & D\\
G & 0
\end{array}\right)$ with 
\[
\mathcal{D}(A^{\ast})=\left\{ (x,y)\in\mathcal{D}(G)\times\mathcal{D}(D)\,\left|\,(\pi_{\mathcal{BD}(G)}x,\stackrel{\bullet}{D}\pi_{\mathcal{BD}(D)}y)\in-h^{\ast}\right.\right\} .
\]
\end{prop}
\begin{proof}
Note that due to the linearity of $h$ we have $(0,0)\in h$ and thus,
\[
\left(\begin{array}{cc}
0 & D_{c}\\
G_{c} & 0
\end{array}\right)\subseteq A.
\]
This shows that $A$ is densely defined. Moreover, we deduce that
\[
A^{\ast}\subseteq-\left(\begin{array}{cc}
0 & D\\
G & 0
\end{array}\right).
\]
Let $(x,y)\in\mathcal{D}(A^{\ast}).$ Then, for all $(u,v)\in\mathcal{D}(A)$
we have 
\[
\left\langle \left.A\left(\begin{array}{c}
u\\
v
\end{array}\right)\right|\left(\begin{array}{c}
x\\
y
\end{array}\right)\right\rangle _{H_{0}\oplus H_{1}}=\left\langle \left.\left(\begin{array}{c}
u\\
v
\end{array}\right)\right|\left(\begin{array}{c}
-Dy\\
-Gx
\end{array}\right)\right\rangle _{H_{0}\oplus H_{1}}.
\]
The left hand side of the latter equation gives
\begin{align*}
\langle Dv|x\rangle_{H_{0}}+\langle Gu|y\rangle_{H_{1}} & =\langle D_{c}\pi_{D_{c}}^{\ast}\pi_{D_{c}}v|x\rangle_{H_{0}}+\langle\pi_{\mathcal{BD}(G)}^{\ast}\stackrel{\bullet}{D}\pi_{\mathcal{BD}(D)}v|x\rangle_{H_{0}}\\
 & \quad+\langle G_{c}\pi_{G_{c}}^{\ast}\pi_{G_{c}}u|y\rangle_{H_{1}}+\langle\pi_{\mathcal{BD}(D)}^{\ast}\stackrel{\bullet}{G}\pi_{\mathcal{BD}(G)}u|y\rangle_{H_{1}}\\
 & =-\langle\pi_{D_{c}}^{\ast}\pi_{D_{c}}v|Gx\rangle_{H_{1}}+\langle\pi_{\mathcal{BD}(G)}^{\ast}\stackrel{\bullet}{D}\pi_{\mathcal{BD}(D)}v|x\rangle_{H_{0}}\\
 & \quad-\langle\pi_{G_{c}}^{\ast}\pi_{G_{c}}u|Dy\rangle_{H_{0}}+\langle\pi_{\mathcal{BD}(D)}^{\ast}\stackrel{\bullet}{G}\pi_{\mathcal{BD}(G)}u|y\rangle_{H_{1}}
\end{align*}
On the other hand 
\begin{align*}
\langle u|-Dy\rangle_{H_{0}}+\langle v|-Gx\rangle_{H_{1}} & =-\langle\pi_{G_{c}}^{\ast}\pi_{G_{c}}u|Dy\rangle_{H_{0}}-\langle\pi_{\mathcal{BD}(G)}^{\ast}\pi_{\mathcal{BD}(G)}u|Dy\rangle_{H_{0}}\\
 & \quad-\langle\pi_{D_{c}}^{\ast}\pi_{D_{c}}v|Gx\rangle_{H_{1}}-\langle\pi_{\mathcal{BD}(D)}^{\ast}\pi_{\mathcal{BD}(D)}v|Gx\rangle_{H_{1}}.
\end{align*}
Thus, we end up with 
\begin{align*}
 & \langle\pi_{\mathcal{BD}(G)}^{\ast}\stackrel{\bullet}{D}\pi_{\mathcal{BD}(D)}v|x\rangle_{H_{0}}+\langle\pi_{\mathcal{BD}(D)}^{\ast}\stackrel{\bullet}{G}\pi_{\mathcal{BD}(G)}u|y\rangle_{H_{1}}\\
 & =-\left(\langle\pi_{\mathcal{BD}(D)}^{\ast}\pi_{\mathcal{BD}(D)}v|Gx\rangle_{H_{1}}+\langle\pi_{\mathcal{BD}(G)}^{\ast}\pi_{\mathcal{BD}(G)}u|Dy\rangle_{H_{0}}\right)\\
 & =-\left(\langle G\pi_{\mathcal{BD}(G)}^{\ast}\stackrel{\bullet}{D}\pi_{\mathcal{BD}(D)}v|Gx\rangle_{H_{1}}+\langle D\pi_{\mathcal{BD}(D)}^{\ast}\stackrel{\bullet}{G}\pi_{\mathcal{BD}(G)}u|Dy\rangle_{H_{0}}\right),
\end{align*}
which yields 
\[
\langle\pi_{\mathcal{BD}(G)}^{\ast}\stackrel{\bullet}{D}\pi_{\mathcal{BD}(D)}v|x\rangle_{H^{1}(|G|+\i)}=-\langle\pi_{\mathcal{BD}(D)}^{\ast}\stackrel{\bullet}{G}\pi_{\mathcal{BD}(G)}u|y\rangle_{H^{1}(|D|+\i)}.
\]
Hence, 
\[
\langle\stackrel{\bullet}{D}\pi_{\mathcal{BD}(D)}v|\pi_{\mathcal{BD}(G)}x\rangle_{\mathcal{BD}(G)}=\langle\pi_{\mathcal{BD}(G)}u|-\stackrel{\bullet}{D}\pi_{\mathcal{BD}(D)}y\rangle_{\mathcal{BD}(G)},
\]
which yields, using \prettyref{rem:adjoint}, 
\[
(\pi_{\mathcal{BD}(G)}x,-\stackrel{\bullet}{D}\pi_{\mathcal{BD}(D)}y)\in h^{\ast}.
\]
Assume now that $(x,y)\in\mathcal{D}(G)\times\mathcal{D}(D)$ with
$(\pi_{\mathcal{BD}(G)}x,-\stackrel{\bullet}{D}\pi_{\mathcal{BD}(D)}y)\in h^{\ast}$.
Then, for each $(u,v)\in\mathcal{D}(A)$ we compute 
\begin{align*}
\left\langle \left.A\left(\begin{array}{c}
u\\
v
\end{array}\right)\right|\left(\begin{array}{c}
x\\
y
\end{array}\right)\right\rangle _{H_{0}\oplus H_{1}} & =\langle Dv|x\rangle_{H_{0}}+\langle Gu|y\rangle_{H_{1}}\\
 & =-\langle\pi_{D_{c}}^{\ast}\pi_{D_{c}}v|Gx\rangle_{H_{1}}+\langle\pi_{\mathcal{BD}(G)}^{\ast}\stackrel{\bullet}{D}\pi_{\mathcal{BD}(D)}v|x\rangle_{H_{0}}\\
 & \quad-\langle\pi_{G_{c}}^{\ast}\pi_{G_{c}}u|Dy\rangle_{H_{0}}+\langle\pi_{\mathcal{BD}(D)}^{\ast}\stackrel{\bullet}{G}\pi_{\mathcal{BD}(G)}u|y\rangle_{H_{1}}\\
 & =-\langle v|Gx\rangle_{H_{1}}+\langle\pi_{\mathcal{BD}(D)}^{\ast}\pi_{\mathcal{BD}(D)}v|Gx\rangle_{H_{1}}+\langle\pi_{\mathcal{BD}(G)}^{\ast}\stackrel{\bullet}{D}\pi_{\mathcal{BD}(D)}v|x\rangle_{H_{0}}\\
 & \quad-\langle u|Dy\rangle_{H_{0}}+\langle\pi_{\mathcal{BD}(G)}^{\ast}\pi_{\mathcal{BD}(G)}u|Dy\rangle_{H_{0}}+\langle\pi_{\mathcal{BD}(D)}^{\ast}\stackrel{\bullet}{G}\pi_{\mathcal{BD}(G)}u|y\rangle_{H_{1}}\\
 & =-\langle v|Gx\rangle_{H_{1}}+\langle\stackrel{\bullet}{D}\pi_{\mathcal{BD}(D)}v|\pi_{\mathcal{BD}(G)}x\rangle_{\mathcal{BD}(G)}\\
 & \quad-\langle u|Dy\rangle_{H_{0}}-\langle\pi_{\mathcal{BD}(G)}u|-\stackrel{\bullet}{D}\pi_{\mathcal{BD}(D)}y\rangle_{\mathcal{BD}(G)}\\
 & =\left\langle \left.\left(\begin{array}{c}
u\\
v
\end{array}\right)\right|-\left(\begin{array}{cc}
0 & D\\
G & 0
\end{array}\right)\left(\begin{array}{c}
x\\
y
\end{array}\right)\right\rangle _{H_{0}\oplus H_{1}},
\end{align*}
where we again have used \prettyref{rem:adjoint}. This completes
the proof.
\end{proof}
As a consequence of our considerations above, we obtain the following
corollary.
\begin{cor}
\label{cor:skew-selfadjoint}The operator $A$ is skew-selfadjoint
if and only if there exists a skew-selfadjoint relation $h\subseteq\mathcal{BD}(G)\oplus\mathcal{BD}(G)$
such that the domain of $A$ is given by \prettyref{eq:D(A)}.
\end{cor}

\section{Classical trace spaces}

In this section we compare the classical trace spaces $H^{\pm\frac{1}{2}}(\partial\Omega)$
with the abstract boundary data spaces $\mathcal{BD}(\grad)$ and
$\mathcal{BD}(\dive),$ where $\grad$ and $\dive$ are defined as
in Section 2. Throughout this section we assume that $\Omega\subseteq\mathbb{R}^{n}$
is a bounded Lipschitz domain.\\
We recall the definition of the classical trace spaces $H^{\frac{1}{2}}(\partial\Omega)$
and $H^{-\frac{1}{2}}(\partial\Omega)$.
\begin{prop}[{\cite[Theorem 1.2]{necas2011direct}}]
 The operator 
\begin{align*}
\gamma_{D}:C^{\infty}(\overline{\Omega})\subseteq H^{1}(|\grad|+\i) & \to L_{2}(\partial\Omega)\\
u & \mapsto u|_{\partial\Omega}
\end{align*}
is bounded and thus, it has a unique bounded extension to $H^{1}(|\grad|+\i)$,
which will be again denoted by $\gamma_{D}$.\end{prop}
\begin{defn*}
We set $H^{\frac{1}{2}}(\partial\Omega)\coloneqq\gamma_{D}[H^{1}(|\grad|+\i)]$
and equip this space with the norm
\[
|u|_{H^{\frac{1}{2}}(\partial\Omega)}\coloneqq\sqrt{|u|_{L_{2}(\partial\Omega)}^{2}+\intop_{\partial\Omega}\intop_{\partial\Omega}\frac{|u(x)-u(y)|^{2}}{\left|x-y\right|^{n}}\mbox{ d}y\mbox{ d}x}.
\]
Moreover, we set $H^{-\frac{1}{2}}(\partial\Omega)\coloneqq H^{\frac{1}{2}}(\partial\Omega)^{\ast},$
the dual space of $H^{\frac{1}{2}}(\partial\Omega).$\end{defn*}
\begin{rem}
\label{rem:embeddings}One can show that 
\begin{align*}
\iota:H^{\frac{1}{2}}(\partial\Omega) & \to L_{2}(\partial\Omega)\\
u & \mapsto u
\end{align*}
is bounded and has dense range, see \cite[Theorem 4.9]{necas2011direct}
(indeed, one can even show the compactness of this embedding, \cite[Theorem 6.2]{necas2011direct}).
Consequently, one obtains that 
\begin{align*}
\iota':L_{2}(\partial\Omega) & \to H^{-\frac{1}{2}}(\partial\Omega)\\
f & \mapsto\left(H^{\frac{1}{2}}(\partial\Omega)\ni u\mapsto\langle f|\iota u\rangle_{L_{2}(\partial\Omega)}\right)
\end{align*}
is also bounded with dense range (and even compact).\end{rem}
\begin{prop}[{\cite[Theorems 4.10, 5.5, 5.7]{necas2011direct}}]
\label{prop:properties_traces} The operator $\gamma_{D}:H^{1}(|\grad|+\i)\to H^{\frac{1}{2}}(\partial\Omega)$
is bounded and $\mathcal{N}(\gamma_{D})=H^{1}(|\grad_{c}|+\i)$. Moreover,
there exists a bounded right inverse, i.e. there exists a bounded
linear operator 
\[
E:H^{\frac{1}{2}}(\partial\Omega)\to H^{1}(|\grad|+\i)
\]
such that $\gamma_{D}\circ E=1_{H^{\frac{1}{2}}(\partial\Omega)},$
the identity on $H^{\frac{1}{2}}(\partial\Omega)$. 
\end{prop}
With the help of the last proposition we can show that $\mathcal{BD}(\grad)$
and $H^{\frac{1}{2}}(\partial\Omega)$ are isomorphic.
\begin{cor}
\label{cor:BD(grad)}The operator $\gamma\coloneqq\gamma_{D}\circ\pi_{\mathcal{BD}(\grad)}^{\ast}:\mathcal{BD}(\grad)\to H^{\frac{1}{2}}(\partial\Omega)$
is a Banach space isomorphism.\end{cor}
\begin{proof}
That $\gamma$ is one-to-one and bounded follows from \prettyref{prop:properties_traces}.
To see that $\gamma$ is onto, let $u\in H^{\frac{1}{2}}(\partial\Omega)$.
Then we set $v\coloneqq\pi_{\mathcal{BD}(\grad)}Eu$ and obtain
\begin{align*}
u & =\gamma_{D}(Eu)\\
 & =\gamma_{D}(\pi_{\grad_{c}}^{\ast}\pi_{\grad_{c}}Eu+\pi_{\mathcal{BD}(\grad)}^{\ast}\pi_{\mathcal{BD}(\grad)}Eu)\\
 & =\gamma_{D}(\pi_{\mathcal{BD}(\grad)}^{\ast}\pi_{\mathcal{BD}(\grad)}Eu)=\gamma(v).
\end{align*}
Moreover, 
\[
|v|_{\mathcal{BD}(\grad)}=|\pi_{\mathcal{BD}(\grad)}^{\ast}v|_{H^{1}(|\grad|+\i)}\leq\|E\||u|_{H^{\frac{1}{2}}(\partial\Omega)},
\]
which shows the continuity of $\gamma^{-1}.$
\end{proof}
Using this observation, we may define an alternative, but equivalent,
norm on $H^{\frac{1}{2}}(\partial\Omega)$ by 
\[
|u|\coloneqq|\gamma^{-1}u|_{\mathcal{BD}(\grad)}\quad(u\in H^{\frac{1}{2}}(\partial\Omega)).
\]
Using this norm, the operator $\gamma:\mathcal{BD}(\grad)\to H^{\frac{1}{2}}(\partial\Omega)$
becomes unitary. In the subsequent part we will always assume that
$H^{\frac{1}{2}}(\partial\Omega)$ is equipped with this equivalent
norm. \\
In order to deal with normal derivatives, we need the following representation
of the dual of $\mathcal{BD}(\grad).$
\begin{lem}
\label{lem:TR(div)}Let%
\footnote{Recall that according to \prettyref{prop:modulus} the operator $\dive_{c}$
is bounded as an operator from $L_{2}(\Omega)^{n}$ to $H^{-1}(|\dive_{c}^{\ast}|+\i)=H^{-1}(|\grad|+\i)$.%
} $\mathcal{TR}(\dive)\coloneqq\left(\dive-\dive_{c}\right)\left[H^{1}(|\dive|+\i)\right]\subseteq H^{-1}(|\grad|+\i).$
Then, the mapping $\Phi:\mathcal{TR}(\dive)\to\mathcal{BD}(\grad)'$
given by 
\[
\Phi\left(\left(\dive-\dive_{c}\right)v\right)(u)\coloneqq\left\langle \left.\left(\dive-\dive_{c}\right)v\right|\pi_{\mathcal{BD}(\grad)}^{\ast}u\right\rangle _{H^{-1}(|\grad|+\i),H^{1}(|\grad|+\i)}
\]
is unitary.\end{lem}
\begin{proof}
Since the functional $(\dive-\dive_{c})v$ vanishes on $\mathcal{BD}(\grad)^{\bot}=H^{1}(|\grad_{c}|+\i)$
we easily see that $\Phi$ is isometric. To show the surjectivity
of $\Phi$ we take $\varphi\in\mathcal{BD}(\grad)'.$ Then there exists
$\tilde{u}\in\mathcal{BD}(\grad)$ such that 
\begin{align*}
\varphi(u) & =\langle\tilde{u}|u\rangle_{\mathcal{BD}(\grad)}\\
 & =\langle\grad\pi_{\mathcal{BD}(\grad)}^{\ast}\tilde{u}|\grad\pi_{\mathcal{BD}(\grad)}^{\ast}u\rangle_{L_{2}(\Omega)^{n}}+\langle\pi_{\mathcal{BD}(\grad)}^{\ast}\tilde{u}|\pi_{\mathcal{BD}(\grad)}^{\ast}u\rangle_{L_{2}(\Omega)}\\
 & =\langle\pi_{\mathcal{BD}(\dive)}^{\ast}\stackrel{\bullet}{\grad}\tilde{u}|\grad\pi_{\mathcal{BD}(\grad)}^{\ast}u\rangle_{L_{2}(\Omega)^{n}}+\langle\dive\pi_{\mathcal{BD}(\dive)}^{\ast}\stackrel{\bullet}{\grad}\tilde{u}|\pi_{\mathcal{BD}(\grad)}^{\ast}u\rangle_{L_{2}(\Omega)}\\
 & =\Phi\left((\dive-\dive_{c})\pi_{\mathcal{BD}(\dive)}^{\ast}\stackrel{\bullet}{\grad}\tilde{u}\right)(u)
\end{align*}
for every $u\in\mathcal{BD}(\grad),$ which proves that $\Phi$ is
onto.\end{proof}
\begin{cor}
The spaces $H^{-\frac{1}{2}}(\partial\Omega)$ and $\mathcal{TR}(\dive)$
are isomorphic via the mapping $\Phi^{\ast}\circ\gamma'$, where $\gamma'$
denotes the dual mapping of $\gamma$, given in \prettyref{cor:BD(grad)}.
\end{cor}
For a function $v\in H^{1}(|\dive|+\i)$ one can define the boundary
term $v\cdot N\in H^{-\frac{1}{2}}(\partial\Omega)$, where $N$ denotes
the unit outward normal vector field on $\partial\Omega$ (which exists
according to \cite[Lemma 4.2]{necas2011direct}) via Green's formula%
\footnote{Note that $\gamma^{\ast}=\gamma^{-1}:H^{\frac{1}{2}}(\partial\Omega)\to\mathcal{BD}(\grad)$
according to our choice of the norm on $H^{\frac{1}{2}}(\partial\Omega)$.%
}: 
\begin{align}
\langle v\cdot N|u\rangle_{H^{-\frac{1}{2}}(\partial\Omega),H^{\frac{1}{2}}(\partial\Omega)} & =\langle\dive v|E(u)\rangle_{L_{2}(\Omega)}+\langle v|\grad E(u)\rangle_{L_{2}(\Omega)^{n}}\label{eq:Green}\\
 & =\langle\dive v|\pi_{\mathcal{BD}(\grad)}^{\ast}\gamma^{\ast}u\rangle_{L_{2}(\Omega)}+\langle v|\grad\pi_{\mathcal{BD}(\grad)}^{\ast}\gamma^{\ast}u\rangle_{L_{2}(\Omega)^{n}}\nonumber \\
 & =\left\langle \left.\left(\dive-\dive_{c}\right)v\right|\pi_{\mathcal{BD}(\grad)}^{\ast}\gamma^{\ast}u\right\rangle _{H^{-1}(|\grad|+\i),H^{1}(|\grad|+\i)}\nonumber \\
 & =\Phi\left((\dive-\dive_{c})v\right)(\gamma^{\ast}u).\nonumber 
\end{align}

Finally, we recall a result from \cite{Picard2012_boundary_control}.
\begin{prop}[{\cite[Theorem 4.5]{Picard2012_boundary_control}}]
\label{prop:BD(div) and TR(div)} The operator 
\[
\tilde{\gamma}\coloneqq(\dive-\dive_{c})\pi_{\mathcal{BD}(\dive)}^{\ast}:\mathcal{BD}(\dive)\to\mathcal{TR}(\dive)
\]
is unitary.
\end{prop}

\subsection*{Some classical boundary conditions}

This subsection is devoted to the study of classical boundary conditions
within the framework of abstract boundary data spaces. Moreover, we
discuss which boundary conditions yield a maximal monotone realization
of the operator 
\begin{equation}
A\subseteq\left(\begin{array}{cc}
0 & \dive\\
\grad & 0
\end{array}\right).\label{eq:block_op}
\end{equation}

To avoid technicalities, we only treat the most simple cases of such
boundary conditions and refer to the next section for more advanced
examples. We start with inhomogeneous Dirichlet and Neumann boundary
conditions.

\subsubsection*{Dirichlet and Neumann boundary conditions}

Throughout let $(u,v)\in\mathcal{D}(\grad)\times\mathcal{D}(\dive)$.
The inhomogeneous Dirichlet boundary condition reads as 
\[
\gamma_{D}u=f,
\]

for some $f\in H^{\frac{1}{2}}(\partial\Omega).$ The latter is equivalent
to the fact that 
\[
\pi_{\mathcal{BD}(\grad)}u=\gamma^{\ast}f,
\]
by \prettyref{cor:BD(grad)} and thus, the boundary relation $h\subseteq\mathcal{BD}(\grad)\oplus\mathcal{BD}(\grad)$
may be given by 
\[
h\coloneqq\left\{ (\gamma^{\ast}f,y)\,|\, y\in\mathcal{BD}(\grad)\right\} .
\]
Obviously, this relation is maximal monotone and hence, the operator
$A$ with domain 
\begin{align*}
\mathcal{D}(A) & =\left\{ (u,v)\in\mathcal{D}(\grad)\times\mathcal{D}(\dive)\,\left|\,(\pi_{\mathcal{BD}(\grad)}u,\stackrel{\bullet}{\dive}\pi_{\mathcal{BD}(\dive)}v)\in h\right.\right\} \\
 & =\left\{ (u,v)\in\mathcal{D}(\grad)\times\mathcal{D}(\dive)\,\left|\,\pi_{\mathcal{BD}(\grad)}u=\gamma^{\ast}f\right.\right\} 
\end{align*}
is maximal monotone. Note that only in the case of $f=0,$ the operator
is skew-selfadjoint. \\
In the same way one might deal with Neumann boundary conditions, given
by 
\[
v\cdot N=g
\]
for some $g\in H^{-\frac{1}{2}}(\partial\Omega).$ Using \prettyref{eq:Green}
the latter means that for all $u\in H^{\frac{1}{2}}(\partial\Omega)$
we have that 
\[
\langle g|u\rangle_{H^{-\frac{1}{2}}(\partial\Omega),H^{\frac{1}{2}}(\partial\Omega)}=\Phi\left((\dive-\dive_{c})v\right)(\gamma^{\ast}u)
\]
or equivalently for all $w\in\mathcal{BD}(\grad)$
\[
\left(\gamma'g\right)(w)=\langle g|\gamma w\rangle_{H^{-\frac{1}{2}}(\partial\Omega),H^{\frac{1}{2}}(\partial\Omega)}=\Phi\left((\dive-\dive_{c})v\right)(w).
\]
Since $(\dive-\dive_{c})v=(\dive-\dive_{c})\pi_{\mathcal{BD}(\dive)}^{\ast}\pi_{\mathcal{BD}(\dive)}v,$
the latter means 
\[
\gamma'g=\Phi\tilde{\gamma}\pi_{\mathcal{BD}(\dive)}v
\]
or equivalently
\[
\pi_{\mathcal{BD}(\dive)}v=\tilde{\gamma}^{\ast}\Phi^{\ast}\left(\gamma'g\right).
\]
 using \prettyref{lem:TR(div)} and \prettyref{prop:BD(div) and TR(div)}.
Thus, the boundary relation $h$ is given by 
\[
h\coloneqq\left\{ \left.\left(x,\stackrel{\bullet}{\dive}\tilde{\gamma}^{\ast}\Phi^{\ast}\left(\gamma'g\right)\right)\,\right|\, x\in\mathcal{BD}(\grad)\right\} ,
\]
which is again maximal monotone and thus, the operator $A$ with domain
\begin{align*}
\mathcal{D}(A) & =\left\{ (u,v)\in\mathcal{D}(\grad)\times\mathcal{D}(\dive)\,\left|\,(\pi_{\mathcal{BD}(\grad)}u,\stackrel{\bullet}{\dive}\pi_{\mathcal{BD}(\dive)}v)\in h\right.\right\} \\
 & =\left\{ (u,v)\in\mathcal{D}(\grad)\times\mathcal{D}(\dive)\,\left|\,\stackrel{\bullet}{\dive}\pi_{\mathcal{BD}(\dive)}v=\stackrel{\bullet}{\dive}\tilde{\gamma}^{\ast}\Phi^{\ast}\left(\gamma'g\right)\right.\right\} \\
 & =\left\{ (u,v)\in\mathcal{D}(\grad)\times\mathcal{D}(\dive)\,\left|\,\pi_{\mathcal{BD}(\dive)}v=\tilde{\gamma}^{\ast}\Phi^{\ast}\left(\gamma'g\right)\right.\right\} 
\end{align*}
is maximal monotone. Note again that the operator gets skew-selfadjoint
if and only if $g=0$.

\subsubsection*{Robin type boundary conditions}

Let $(u,v)\in\mathcal{D}(\grad)\times\mathcal{D}(\dive)$. A Robin-type
boundary condition may be written as 
\begin{equation}
\iota'\iota\gamma_{D}u-v\cdot N=g,\label{eq:Robin_classical}
\end{equation}
for some $g\in H^{-\frac{1}{2}}(\partial\Omega)$, where $\iota$
denotes the embedding $H^{\frac{1}{2}}(\partial\Omega)\hookrightarrow L_{2}(\partial\Omega)$
and $\iota'$ denotes the embedding $L_{2}(\partial\Omega)\hookrightarrow H^{-\frac{1}{2}}(\partial\Omega)$
as in \prettyref{rem:embeddings}. The boundary condition \prettyref{eq:Robin_classical}
means that for each $w\in\mathcal{BD}(\grad)$ we have that 
\[
\langle\iota\gamma_{D}u|\iota\gamma w\rangle_{L_{2}(\partial\Omega)}-\Phi\left(\tilde{\gamma}\pi_{\mathcal{BD}(\dive)}v\right)(w)=\langle g|\gamma w\rangle_{H^{-\frac{1}{2}}(\partial\Omega),H^{\frac{1}{2}}(\partial\Omega)},
\]
or equivalently 
\begin{align*}
 & \langle\gamma^{\ast}\iota^{\ast}\iota\gamma_{\mathcal{BD}(\grad)}u|w\rangle_{\mathcal{BD}(\grad)}\\
 & =\langle\iota\gamma\pi_{\mathcal{BD}(\grad)}u|\iota\gamma w\rangle_{L_{2}(\partial\Omega)}\\
 & =\left(\gamma'g+\Phi\tilde{\gamma}\pi_{\mathcal{BD}(\dive)}v\right)(w)\\
 & =\Phi\tilde{\gamma}\left(\tilde{\gamma}^{\ast}\Phi^{\ast}\gamma'g+\pi_{\mathcal{BD}(\dive)}v\right)(w)\\
 & =\left\langle \left.(\dive-\dive_{c})\pi_{\mathcal{BD}(\dive)}^{\ast}\left(\tilde{\gamma}^{\ast}\Phi^{\ast}\gamma'g+\pi_{\mathcal{BD}(\dive)}v\right)\right|\pi_{\mathcal{BD}(\grad)}^{\ast}w\right\rangle _{H^{-1}(|\grad|+\i),H^{1}(|\grad|+\i)}\\
 & =\left\langle \left.\pi_{\mathcal{BD}(\grad)}^{\ast}\stackrel{\bullet}{\dive}\left(\tilde{\gamma}^{\ast}\Phi^{\ast}\gamma'g+\pi_{\mathcal{BD}(\dive)}v\right)\right|\pi_{\mathcal{BD}(\grad)}^{\ast}w\right\rangle _{L_{2}(\Omega)}\\
 & \quad+\left\langle \left.\grad\pi_{\mathcal{BD}(\grad)}^{\ast}\stackrel{\bullet}{\dive}\left(\tilde{\gamma}^{\ast}\Phi^{\ast}\gamma'g+\pi_{\mathcal{BD}(\dive)}v\right)\right|\grad\pi_{\mathcal{BD}(\grad)}^{\ast}w\right\rangle _{L_{2}(\Omega)^{n}}\\
 & =\left\langle \left.\stackrel{\bullet}{\dive}\left(\tilde{\gamma}^{\ast}\Phi^{\ast}\gamma'g+\pi_{\mathcal{BD}(\dive)}v\right)\right|w\right\rangle _{\mathcal{BD}(\grad)}.
\end{align*}
This gives 
\[
\gamma^{\ast}\iota^{\ast}\iota\gamma\pi_{\mathcal{BD}(\grad)}u=\stackrel{\bullet}{\dive}\left(\tilde{\gamma}^{\ast}\Phi^{\ast}\gamma'g+\pi_{\mathcal{BD}(\dive)}v\right),
\]
which leads to the following definition of the boundary relation $h$
\[
h\coloneqq\left\{ (x,y)\,\left|\,\gamma^{\ast}\iota^{\ast}\iota\gamma x=\stackrel{\bullet}{\dive}\tilde{\gamma}^{\ast}\Phi^{\ast}\gamma'g+y\right.\right\} .
\]
To see that this relation is maximal monotone, we state the following
trivial observation.
\begin{lem}
Let $H$ be a Hilbert space and $L\subseteq H\oplus H$. Moreover
let $(x,y)\in H\times H.$ Then $L$ is maximal monotone if and only
if 
\[
L+\{(x,y)\}=\left\{ (u+x,v+y)\,|\,(u,v)\in L\right\} 
\]
is maximal monotone.
\end{lem}
The latter gives, that $h$ is maximal monotone if and only if 
\[
h-\left\{ \left(0,\stackrel{\bullet}{\dive}\tilde{\gamma}^{\ast}\Phi^{\ast}\gamma'g\right)\right\} 
\]
is maximal monotone. The latter relation is nothing but the non-negative,
selfadjoint operator $\gamma^{\ast}\iota^{\ast}\iota\gamma$ and thus,
maximal monotone. We note here that a Robin boundary condition of
the form 
\begin{equation}
\iota'\iota\gamma_{D}u+v\cdot N=g\label{eq:Robin_bad}
\end{equation}
does not lead to a maximal monotone relation $h$ and hence not to
a maximal monotone realization of \prettyref{eq:block_op}. Indeed,
a boundary condition of the form \prettyref{eq:Robin_bad} would yield
a relation of the form 
\[
h=\left\{ (x,y)\,\left|\,-\gamma^{\ast}\iota^{\ast}\iota\gamma x=\stackrel{\bullet}{\dive}\tilde{\gamma}^{\ast}\Phi^{\ast}\gamma'g+y\right.\right\} ,
\]
which is not even monotone.
\begin{rem}
In applications it turns out that different boundary conditions are
imposed on different parts of the boundary. We refer the reader to
the next section, where in the concrete case of a contact problem
in visco-elasticity such boundary conditions are studied.
\end{rem}

\section{Examples}

\subsection{Port-Hamiltonian systems}

In this section we study so-called linear Port-Hamiltonian systems.
Originally, these systems were defined in the context of differential
forms in \cite{vanderSchaft2002}. However, we follow the notion given
in \cite{leGorrec2005}, \cite[Chapter 7]{Jacob_Zwart2012}. Throughout,
let $n\in\mathbb{N}$ and $a,b\in\mathbb{R}$ with $a<b.$ Moreover,
let $P_{1}\in\mathbb{C}^{n\times n}$ an invertible, selfadjoint matrix,
$P_{0}\in\mathbb{C}^{n\times n}$ be skew-selfadjoint and $\mathcal{H}\in L_{\infty}([a,b];\mathbb{C}^{n\times n})$
such that $\mathcal{H}(t)$ is selfadjoint and there exists $c>0$
with $\mathcal{H}(t)\geq c$ for almost every $t\in[a,b].$ The differential
operator under consideration is a suitable restriction $A$ of 
\[
P_{1}\partial\mathcal{H}+P_{0}\mathcal{H},
\]
with maximal domain, where $\partial:H^{1}([a,b];\mathbb{C}^{n})\subseteq L_{2}([a,b];\mathbb{C}^{n})\to L_{2}([a,b];\mathbb{C}^{n})$
denotes the usual weak derivative on $L_{2}.$ In particular we want
to characterize those restrictions $A$, which yield a maximal monotone
operator in a suitable Hilbert space. In case of a linear operator
$A$, a class of maximal monotone realizations was given in \cite[Section 4]{leGorrec2005},
\cite[Theorem 7.2.3]{Jacob_Zwart2012} (see also \prettyref{thm:lin_port}
below). 
\begin{lem}
Let $H$ be a Hilbert space and $P\in L(H)$ selfadjoint with $P\geq c>0.$
Moreover, let $A\subseteq H\oplus H$ be a maximal monotone relation.
We denote by $H_{P}$ the Hilbert space $H$ equipped with the weighted
inner product 
\[
\langle x|y\rangle_{H_{P}}\coloneqq\langle Px|y\rangle_{H}\quad(x,y\in H).
\]
Then $AP\coloneqq\left\{ (x,y)\in H_{P}\oplus H_{P}\,|\,(Px,y)\in A\right\} $
is maximal monotone in $H_{P}.$\end{lem}
\begin{proof}
The monotonicity of $AP$ in $H_{P}$ is clear. Moreover, if $(u,v)\in H_{P}\oplus H_{P}$
satisfies 
\[
\Re\langle x-u|y-v\rangle_{H_{P}}\geq0
\]
for each $(x,y)\in AP,$ then 
\[
\Re\langle\tilde{x}-Pu|y-v\rangle_{H}\geq0
\]
for each $(\tilde{x},y)\in A$ and thus, $(Pu,v)\in A.$ The latter
yields $(u,v)\in AP$ and thus, $AP$ is maximal monotone.
\end{proof}
The last lemma shows, that we can assume without loss of generality
that $\mathcal{H}(t)=1$ for each $t\in[a,b]$. For linear maximal
monotone operators $A$, the last lemma was also shown in \cite[Lemma 7.2.2]{Jacob_Zwart2012}
with an alternative proof. Moreover, since $P_{0}$ is skew-selfadjoint,
it suffices to treat the case $P_{0}=0,$ since if $P_{1}\partial+P_{0}$
is maximal monotone then so is $P_{1}\partial$ and vice versa.

Thus, we are led to consider maximal monotone realizations of the
operator $P_{1}\partial$. Although, this operator seems not to be
of the form discussed in Section 3, a simple trick adopted from \cite{Picard2012_mother}
will allow us to write the operator as a block operator matrix of
the form $\left(\begin{array}{cc}
0 & D\\
G & 0
\end{array}\right)$. Without loss of generality we assume that the interval $[a,b]$
is symmetric around $0$, i.e. $a=-b.$ We consider the following
operators.
\begin{defn*}
Let $L_{2,\mathrm{e}}([-b,b];\mathbb{C}^{n})\coloneqq\left\{ f\in L_{2}([-b,b];\mathbb{C}^{n})\,|\, f(x)=f(-x)\quad x\in[-b,b]\mbox{ a.e.}\right\} $
and $L_{2,\mathrm{o}}([-b,b];\mathbb{C}^{n})\coloneqq\left\{ f\in L_{2}([-b,b];\mathbb{C}^{n})\,|\, f(x)=-f(-x)\quad x\in[-b,b]\mbox{ a.e.}\right\} $.
Then clearly, $L_{2,\mathrm{e}}$ and $L_{2,\mathrm{o}}$ are orthogonal
closed subspaces of $L_{2}$ such that 
\[
L_{2}([-b,b];\mathbb{C}^{n})=L_{2,\mathrm{e}}([-b,b];\mathbb{C}^{n})\oplus L_{2,\mathrm{o}}([-b,b];\mathbb{C}^{n}).
\]
We denote the corresponding projectors by $\pi_{\mathrm{e}}$ and
$\pi_{\mathrm{o}},$ respectively. Moreover, we define 
\[
\partial_{\mathrm{e}}:H^{1}([-b,b];\mathbb{C}^{n})\cap L_{2,\mathrm{e}}([-b,b];\mathbb{C}^{n})\subseteq L_{2,\mathrm{e}}([-b,b];\mathbb{C}^{n})\to L_{2,\mathrm{o}}([-b,b];\mathbb{C}^{n})
\]
and 
\[
\partial_{\mathrm{o}}:H^{1}([-b,b];\mathbb{C}^{n})\cap L_{2,\mathrm{o}}([-b,b];\mathbb{C}^{n})\subseteq L_{2,\mathrm{o}}([-b,b];\mathbb{C}^{n})\to L_{2,\mathrm{e}}([-b,b];\mathbb{C}^{n})
\]
as the usual weak derivative restricted to the even and odd functions,
respectively. Consequently, 
\[
\partial=\left(\begin{array}{cc}
\pi_{\mathrm{e}}^{\ast} & \pi_{\mathrm{o}}^{\ast}\end{array}\right)\left(\begin{array}{cc}
0 & \partial_{\mathrm{o}}\\
\partial_{\mathrm{e}} & 0
\end{array}\right)\left(\begin{array}{c}
\pi_{\mathrm{e}}\\
\pi_{\mathrm{o}}
\end{array}\right),
\]
which shows that $\partial$ and $\left(\begin{array}{cc}
0 & \partial_{\mathrm{o}}\\
\partial_{\mathrm{e}} & 0
\end{array}\right)$ are unitarily equivalent.\end{defn*}
\begin{lem}
We set $\partial_{\mathrm{o},c}\coloneqq-\partial_{\mathrm{e}}^{\ast}$
and $\partial_{\mathrm{e},c}\coloneqq-\partial_{\mathrm{o}}^{\ast}$.
Then $\partial_{\mathrm{o},c}\subseteq\partial_{\mathrm{o}}$ and
$\partial_{\mathrm{e},c}\subseteq\partial_{\mathrm{e}}$ with $\mathcal{D}(\partial_{\mathrm{o,}c})=\left\{ u\in\mathcal{D}(\partial_{\mathrm{o}})\,|\, u(-b)=u(b)=0\right\} $
and $\mathcal{D}(\partial_{\mathrm{e},c})=\left\{ u\in\mathcal{D}(\partial_{\mathrm{e}})\,|\, u(-b)=u(b)=0\right\} $.
Moreover, 
\[
\mathcal{BD}(\partial_{\mathrm{o}})=\lspan\left\{ \sinh e_{i}\,|\, i\in\{1,\ldots,n\}\right\} \mbox{ and }\mathcal{BD}(\partial_{\mathrm{e}})=\lspan\left\{ \cosh e_{i}\,|\, i\in\{1,\ldots,n\}\right\} ,
\]
where $e_{i}$ denotes the $i$-th canonical basis vector in $\mathbb{C}^{n}.$\end{lem}
\begin{proof}
The proof is straightforward and we therefore omit it.
\end{proof}
We come back to the operator $P_{1}\partial$, which can be written
as 
\begin{align*}
\left(\begin{array}{cc}
\pi_{\mathrm{e}}^{\ast} & \pi_{\mathrm{o}}^{\ast}\end{array}\right)\left(\begin{array}{cc}
\pi_{\mathrm{e}}P_{1}\pi_{\mathrm{e}}^{\ast} & 0\\
0 & \pi_{\mathrm{o}}P_{1}\pi_{\mathrm{o}}^{\ast}
\end{array}\right)\left(\begin{array}{cc}
0 & \partial_{\mathrm{o}}\\
\partial_{\mathrm{e}} & 0
\end{array}\right)\left(\begin{array}{c}
\pi_{\mathrm{e}}\\
\pi_{\mathrm{o}}
\end{array}\right) & =\left(\begin{array}{cc}
\pi_{\mathrm{e}}^{\ast} & \pi_{\mathrm{o}}^{\ast}\end{array}\right)\left(\begin{array}{cc}
0 & \pi_{\mathrm{e}}P_{1}\pi_{\mathrm{e}}^{\ast}\partial_{\mathrm{o}}\\
\pi_{\mathrm{o}}P_{1}\pi_{\mathrm{o}}^{\ast}\partial_{\mathrm{e}} & 0
\end{array}\right)\left(\begin{array}{c}
\pi_{\mathrm{e}}\\
\pi_{\mathrm{o}}
\end{array}\right)\\
 & =\left(\begin{array}{cc}
\pi_{\mathrm{e}}^{\ast} & \pi_{\mathrm{o}}^{\ast}\end{array}\right)\left(\begin{array}{cc}
0 & \pi_{\mathrm{e}}P_{1}\pi_{\mathrm{e}}^{\ast}\partial_{\mathrm{o}}\\
\partial_{\mathrm{e}}\pi_{\mathrm{e}}P_{1}\pi_{\mathrm{e}}^{\ast} & 0
\end{array}\right)\left(\begin{array}{c}
\pi_{\mathrm{e}}\\
\pi_{\mathrm{o}}
\end{array}\right)
\end{align*}
and the operator $\left(\begin{array}{cc}
0 & \pi_{\mathrm{e}}P_{1}\pi_{\mathrm{e}}^{\ast}\partial_{\mathrm{o}}\\
\partial_{\mathrm{e}}\pi_{\mathrm{e}}P_{1}\pi_{\mathrm{e}}^{\ast} & 0
\end{array}\right)$ now fits into our abstract framework with $D\coloneqq\pi_{\mathrm{e}}P_{1}\pi_{\mathrm{e}}^{\ast}\partial_{\mathrm{o}},\, D_{c}\coloneqq\pi_{\mathrm{e}}P_{1}\pi_{\mathrm{e}}^{\ast}\partial_{\mathrm{o},c}$
and $G\coloneqq\partial_{\mathrm{e}}\pi_{\mathrm{e}}P_{1}\pi_{\mathrm{e}}^{\ast},\, G_{c}\coloneqq\partial_{\mathrm{e},c}\pi_{\mathrm{e}}P_{1}\pi_{\mathrm{e}}^{\ast}$. 
\begin{rem}
\label{rem:n-dimensional BD(G)}We note that $\mathcal{BD}(G)$ and
$\mathcal{BD}(D)$ are both $n$-dimensional spaces. More precisely,
let $\lambda_{1},\ldots,\lambda_{n}$ denote the eigenvalues of the
symmetric and invertible matrix $P_{1}$ counted with multiplicity.
Moreover, we denote by $b_{1},\ldots,b_{n}$ the corresponding pairwise
orthonormal eigenvectors. Then 
\begin{align*}
\mathcal{BD}(G) & =\lspan\left\{ \left.\left(x\mapsto\cosh\left(\lambda_{i}^{-1}x\right)b_{i}\right)\,\right|\, i\in\{1,\ldots,n\}\right\} ,\\
\mathcal{BD}(D) & =\lspan\left\{ \left.\left(x\mapsto\sinh\left(\lambda_{i}^{-1}x\right)b_{i}\right)\,\right|\, i\in\{1,\ldots,n\}\right\} .
\end{align*}
\end{rem}
\begin{thm}
\label{thm:port_ham} Let $A\subseteq P_{1}\partial\mathcal{H}+P_{0}\mathcal{H}$
an arbitrary (possibly nonlinear) restriction and let $G,G_{c},D,D_{c}$
as above. Then $A$ is maximal monotone with respect to the weighted
inner product 
\[
\langle u|v\rangle_{\mathcal{H}}\coloneqq\langle\mathcal{H}u|v\rangle_{L_{2}([-b,b];\mathbb{C}^{n})}\quad(u,v\in L_{2}([-b,b];\mathbb{C}^{n}))
\]
if and only if there exists a maximal monotone relation $h\subseteq\mathcal{BD}(G)\oplus\mathcal{BD}(G)$
such that 
\[
\mathcal{D}(A)=\left\{ \left.u\in L_{2}([-b,b];\mathbb{C}^{n})\,\right|\,\mathcal{H}u\in H^{1}([-b,b];\mathbb{C}^{n}),\:\left(\pi_{\mathcal{BD}(G)}\pi_{\mathrm{e}}\mathcal{H}u,\stackrel{\bullet}{D}\pi_{\mathcal{BD}(D)}\pi_{\mathrm{o}}\mathcal{H}u\right)\in h\right\} .
\]
\end{thm}
\begin{proof}
This is a direct consequence of \prettyref{thm:char_max_mon} and
the considerations above.
\end{proof}
In \cite{Jacob_Zwart2012} we find a characterization for a class
of maximal monotone realizations of $P_{1}\partial_{1}\mathcal{H}+P_{0}\mathcal{H}$
in terms of the so-called boundary flow and boundary effort, defined
as 
\begin{align*}
f_{\partial}(u)\coloneqq & \frac{1}{\sqrt{2}}\left(-P_{1}\mathcal{H}u(b)+P_{1}\mathcal{H}u(-b)\right)
\end{align*}
and 
\[
e_{\partial}(u)\coloneqq\frac{1}{\sqrt{2}}\left(\mathcal{H}u(b)+\mathcal{H}u(-b)\right),
\]
respectively. In the next lemma we provide a formulation of these
terms within our framework.
\begin{lem}
\label{lem:flow and effort}Let $G$ be as above and denote by $\lambda_{1},\ldots,\lambda_{n}$
the eigenvalues of $P_{1}$ (counted with multiplicity) and by $b_{1},\ldots,b_{n}\in\mathbb{C}^{n}$
the corresponding pairwise orthonormal eigenvectors. Define 
\begin{align*}
Q:\mathcal{BD}(G) & \to\mathbb{C}^{n}\\
v & \mapsto\sqrt{2}v(b)
\end{align*}
and $S\in\mathbb{C}^{n\times n}$ via $Sb_{i}\coloneqq\lambda_{i}\tanh(\lambda_{i}^{-1}b)b_{i}$
for $i\in\{1,\ldots,n\}$. Then $S$ is selfadjoint and strictly positive
definite and $\sqrt{S}Q$ is unitary. Moreover, for $\mathcal{H}u\in H^{1}([-b,b];\mathbb{C}^{n})$
we have that
\begin{align*}
e_{\partial}(u) & =Q\left(\pi_{\mathcal{BD}(G)}\pi_{\mathrm{e}}\mathcal{H}u\right)\\
f_{\partial}(u) & =-SQ\left(\stackrel{\bullet}{D}\pi_{\mathcal{BD}(D)}\pi_{\mathrm{o}}\mathcal{H}u\right).
\end{align*}
\end{lem}
\begin{proof}
The fact that $S$ is selfadjoint and strictly positive definite holds,
since 
\[
S=U^{\ast}\left(\begin{array}{cccc}
\lambda_{1}\tanh(\lambda_{1}^{-1}b) & 0 & \cdots & 0\\
0 & \ddots &  & \vdots\\
\vdots &  & \ddots & 0\\
0 & \cdots & 0 & \lambda_{n}\tanh(\lambda_{n}^{-1}b)
\end{array}\right)U,
\]
where $U$ is the unitary matrix defined via $Ue_{i}=b_{i}$ for each
$i\in\{1,\ldots,n\}.$ Moreover, for $v\in\mathcal{BD}(G)$ we find
a representation $v=\sum_{i=1}^{n}c_{i}\cosh(\lambda_{i}^{-1}\cdot)b_{i}$
for suitable constants $c_{1},\ldots,c_{n}\in\mathbb{C}$. Thus, using
$\stackrel{\bullet}{G}v=\sum_{i=1}^{n}c_{i}\sinh(\lambda_{i}^{-1}\cdot)b_{i}$
and integration by parts we obtain 
\begin{align*}
|v|_{\mathcal{BD}(G)}^{2} & =\sum_{i=1}^{n}|c_{i}|^{2}\left(\intop_{-b}^{b}\sinh(\lambda_{i}^{-1}x)^{2}\mbox{ d}x+\intop_{-b}^{b}\cosh(\lambda_{i}^{-1}x)^{2}\mbox{ d}x\right)\\
 & =2\sum_{i=1}^{n}|c_{i}|^{2}\lambda_{i}\sinh(\lambda_{i}^{-1}b)\cosh(\lambda_{i}^{-1}b)\\
 & =2\langle\sum_{i=1}^{n}c_{i}\lambda_{i}\sinh(\lambda_{i}^{-1}b)b_{i}|\sum_{i=1}^{n}c_{i}\cosh(\lambda_{i}^{-1}b)b_{i}\rangle_{\mathbb{C}^{n}}\\
 & =\langle SQv|Qv\rangle_{\mathbb{C}^{n}}\\
 & =|\sqrt{S}Qv|_{\mathbb{C}^{n}}^{2}.
\end{align*}
Finally, for $\mathcal{H}u\in H^{1}([-b,b];\mathbb{C}^{n})$ we compute
\begin{align*}
e_{\partial}(u) & =\frac{1}{\sqrt{2}}\left(\mathcal{H}u(b)+\mathcal{H}u(-b)\right)\\
 & =\sqrt{2}\left(\pi_{\mathrm{e}}\mathcal{H}u\right)(b)\\
 & =\sqrt{2}\left(\pi_{\mathcal{BD}(G)}\pi_{\mathrm{e}}\mathcal{H}u\right)(b)\\
 & =Q\pi_{\mathcal{BD}(G)}\pi_{\mathrm{e}}\mathcal{H}u,
\end{align*}
as well as 
\begin{align*}
f_{\partial}(u) & =\frac{1}{\sqrt{2}}\left(-P_{1}\mathcal{H}u(b)+P_{1}\mathcal{H}u(-b)\right)\\
 & =-\sqrt{2}P_{1}\left(\pi_{\mathrm{o}}\mathcal{H}u\right)(b)\\
 & =-\sqrt{2}P_{1}\left(\pi_{\mathcal{BD}(D)}\pi_{\mathrm{o}}\mathcal{H}u\right)(b).
\end{align*}
Using that there exist $c_{1},\ldots c_{n}\in\mathbb{C}$ such that
$\pi_{\mathcal{BD}(D)}\pi_{\mathrm{o}}\mathcal{H}u=\sum_{i=1}^{n}c_{i}\sinh(\lambda_{i}^{-1}\cdot)b_{i}$
we get that 
\[
f_{\partial}(u)=-\sqrt{2}\sum_{i=1}^{n}c_{i}\lambda_{i}\sinh(\lambda_{i}^{-1}b)b_{i}.
\]
On the other hand we have 
\[
\stackrel{\bullet}{D}\pi_{\mathcal{BD}(D)}\pi_{\mathrm{o}}\mathcal{H}u=\sum_{i=1}^{n}c_{i}\cosh(\lambda_{i}^{-1}\cdot)b_{i}
\]
from which we derive that 
\[
f_{\partial}(u)=-SQ\left(\stackrel{\bullet}{D}\pi_{\mathcal{BD}(D)}\pi_{\mathrm{o}}\mathcal{H}u\right).\tag*{\qedhere}
\]

\end{proof}
We now recall a part of \cite[Theorem 7.2.3]{Jacob_Zwart2012} and
provide an alternative proof within our framework.
\begin{thm}
\label{thm:lin_port}Let $A\subseteq P_{1}\partial\mathcal{H}+P_{0}\mathcal{H}$
a linear restriction and let $G,G_{c},D,D_{c}$ as above. If there
exists a matrix $V\in\mathbb{C}^{n\times n}$ such that $V^{\ast}V\leq1$
and 
\[
\mathcal{D}(A)=\left\{ u\in L_{2}([-b,b];\mathbb{C}^{n})\,\left|\,\mathcal{H}u\in H^{1}([-b,b];\mathbb{C}^{n}),\:\left(\begin{array}{cc}
1+V & \;1-V\end{array}\right)\left(\begin{array}{c}
f_{\partial}(u)\\
e_{\partial}(u)
\end{array}\right)=0\right.\right\} ,
\]
then $A$ defines a maximal monotone operator with respect to the
weighted inner product
\[
\langle u|v\rangle_{\mathcal{H}}\coloneqq\langle\mathcal{H}u|v\rangle_{L_{2}([-b,b];\mathbb{C}^{n})}\quad(u,v\in L_{2}([-b,b];\mathbb{C}^{n})).
\]
\end{thm}
\begin{proof}
According to \prettyref{lem:flow and effort} we may rewrite the domain
of $A$ as
\begin{multline*}
\mathcal{D}(A)=\bigg\{ u\in L_{2}([-b,b];\mathbb{C}^{n})\,\bigg|\,\mathcal{H}u\in H^{1}([-b,b];\mathbb{C}^{n}),\\
\:\left(\begin{array}{cc}
1+V & \;1-V\end{array}\right)\left(\begin{array}{cc}
0 & -S\\
1 & 0
\end{array}\right)\left(\begin{array}{c}
Q\pi_{\mathcal{BD}(G)}\pi_{\mathrm{e}}\mathcal{H}u\\
Q\stackrel{\bullet}{D}\pi_{\mathcal{BD}(D)}\pi_{\mathrm{o}}\mathcal{H}u
\end{array}\right)=0\bigg\}.
\end{multline*}
By \prettyref{thm:port_ham} it suffices to check whether 
\begin{equation}
h\coloneqq\left\{ (x,y)\in\mathcal{BD}(G)\oplus\mathcal{BD}(G)\,\left|\,\left(\begin{array}{cc}
1+V & \;1-V\end{array}\right)\left(\begin{array}{cc}
0 & -S\\
1 & 0
\end{array}\right)\left(\begin{array}{c}
Qx\\
Qy
\end{array}\right)=0\right.\right\} \label{eq:h_kernel}
\end{equation}
defines a maximal monotone relation. According to \cite[Lemma 7.3.2]{Jacob_Zwart2012}
the kernel of $\left(\begin{array}{cc}
1+V & \;1-V\end{array}\right)$ equals the set 
\[
\left\{ \left(\left(1-V\right)\ell,(-1-V)\ell\right)\,|\,\ell\in\mathbb{C}^{n}\right\} .
\]
Let $(x,y)\in h.$ Then there exists $\ell\in\mathbb{C}^{n}$ such
that $-SQy=\ell-V\ell$ and $Qx=-\ell-V\ell.$ The latter implies,
using \prettyref{lem:flow and effort} and $V^{\ast}V\leq1$ 
\begin{align*}
\Re\langle x|y\rangle_{\mathcal{BD}(G)} & =\Re\langle Q^{-1}\left(-\ell-V\ell\right)|\left(-SQ\right)^{-1}\left(\ell-V\ell\right)\rangle_{\mathcal{BD}(G)}\\
 & =-\Re\langle Q^{-1}\left(-\ell-V\ell\right)|\left(\sqrt{S}Q\right)^{-1}\sqrt{S^{-1}}\left(\ell-V\ell\right)\rangle_{\mathcal{BD}(G)}\\
 & =-\Re\langle\sqrt{S}\left(-\ell-V\ell\right)|\sqrt{S^{-1}}\left(\ell-V\ell\right)\rangle_{\mathbb{C}^{n}}\\
 & =-\Re\langle\left(-\ell-V\ell\right)|\left(\ell-V\ell\right)\rangle_{\mathbb{C}^{n}}\\
 & =-\Re\langle\left(V^{\ast}V-1\right)\ell|\ell\rangle_{\mathbb{C}^{n}}\geq0.
\end{align*}
Since $h$ is linear, this gives the monotonicity of $h$. For showing
the maximality of $h$ we take $f\in\mathcal{BD}(G).$ We have to
find an element $\ell\in\mathbb{C}^{n}$ such that 
\[
Q^{-1}(-\ell-V\ell)-Q^{-1}S^{-1}\left(\ell-V\ell\right)=f,
\]
which is equivalent to the fact that 
\[
\left(-S(1+V)-(1-V)\right)\ell=SQf.
\]
To show the existence of such a vector $\ell\in\mathbb{C}^{n}$ it
suffices to prove that 
\[
-S(1+V)-(1-V)
\]
is injective. Take $\ell\in\mathbb{C}^{n}$ such that $-S(1+V)\ell-(1-V)\ell=0$.
Then we estimate
\begin{align*}
0 & \leq\langle S(1+V)\ell|(1+V)\ell\rangle_{\mathbb{C}^{n}}\\
 & =-\langle(1-V)\ell|(1+V)\ell\rangle_{\mathbb{C}^{n}}\\
 & =-\langle(1-V^{\ast}V)\ell|\ell\rangle_{\mathbb{C}^{n}}\leq0
\end{align*}
and hence, $(1+V)\ell=0.$ Therefore $(1-V)\ell=-S(1+V)\ell=0$ and
hence, we conclude that $\ell=\frac{1}{2}\left((1+V)\ell+(1-V)\ell\right)=0.$ 
\end{proof}

\subsection{Frictional boundary conditions}

In the context of contact problems in visco-elasticity we find so-called
frictional boundary conditions, which should hold on the part of the
boundary where the contact occurs. Examples for such boundary conditions
can be found for instance in \cite[Section 5]{Migorski_2009}, \cite[p. 171 ff.]{Sofonea2009}
and \cite[134 ff.]{duvaut1976inequalities}.

We follow the model presented in \cite{Migorski_2009}, which was
already discussed by the author in \cite{Trostorff2012_nonlin_bd}
for the case where the frictional boundary condition holds on the
whole boundary. The equations read as follows 
\begin{align*}
\partial_{0}^{2}\rho u-\Div T & =f,\\
T & =C\Grad u+D\Grad\partial_{0}u,
\end{align*}
where $\partial_{0}$ stands for the temporal derivative, $\Grad$
denotes the symmetrized gradient and $\Div$ the row-wise divergence
of a matrix with respect to the spatial variables (the precise definition
will be given below). Following \cite{Trostorff2012_nonlin_bd}, the
latter system can be reformulated as an equation of the form 
\[
\left(\partial_{0}\mathcal{M}+\left(\begin{array}{cc}
0 & \Div\\
\Grad & 0
\end{array}\right)\right)\left(\begin{array}{c}
\partial_{0}u\\
-T
\end{array}\right)=\left(\begin{array}{c}
f\\
0
\end{array}\right)
\]
for a suitable operator $\mathcal{M}$. Throughout, we assume that
$\Omega\subseteq\mathbb{R}^{3}$ is a bounded Lipschitz-domain. Let
$\Gamma_{1},\Gamma_{2},\Gamma_{3}\subseteq\partial\Omega$ be three
pairwise disjoint, measurable sets such that $\Gamma_{1}\cup\Gamma_{2}\cup\Gamma_{3}=\partial\Omega.$
Following \cite{Migorski_2009} we impose the following boundary conditions
\begin{align}
\partial_{0}u & =f_{1}\mbox{ on }\Gamma_{1},\label{eq:Dirichlet}\\
-T\cdot N & =f_{2}\mbox{ on }\Gamma_{2},\label{eq:Neumann}\\
(\partial_{0}u,-T\cdot N) & \in g\mbox{ on }\Gamma_{3},\label{eq:frictional_bd_cd}
\end{align}
for given functions $f_{1}\in L_{2}(\Gamma_{1})^{3},f_{2}\in L_{2}(\Gamma_{2})^{3}$
and a binary relation $g\subseteq L_{2}(\Gamma_{3})^{3}\oplus L_{2}(\Gamma_{3})^{3}.$
Since we are interested in maximal monotone realizations of the block
operator matrix $\left(\begin{array}{cc}
0 & \Div\\
\Grad & 0
\end{array}\right)$, we restrict ourselves to maximal monotone relations $g$ (in \cite{Migorski_2009}
also a class of non-monotone relations was discussed).

We define the operators involved:
\begin{defn*}
We denote by $L_{2}(\Omega)^{3\times3}$ the space of $3\times3$-matrix-valued
$L_{2}(\Omega)$ functions, equipped with the inner product 
\[
\langle T|S\rangle_{L_{2}(\Omega)^{3\times3}}\coloneqq\intop_{\Omega}\trace\left(T(t)^{\ast}S(t)\right)\mbox{ d}t.
\]
Moreover, we denote by $L_{2,\mathrm{sym}}(\Omega)^{3\times3}$ the
closed subspace of symmetric-matrix-valued functions. We define the
operator $\Grad_{c}$ as the closure of 
\begin{align*}
C_{c}^{\infty}(\Omega)^{3}\subseteq L_{2}(\Omega)^{3} & \to L_{2,\mathrm{sym}}(\Omega)^{3\times3}\\
(\phi_{i})_{i\in\{1,2,3\}} & \mapsto\left(\frac{\partial_{i}\phi_{j}+\partial_{j}\phi_{i}}{2}\right)_{i,j\in\{1,2,3\}}
\end{align*}
and the operator $\Div_{c}$ as the closure of 
\begin{align*}
C_{c}^{\infty}(\Omega)^{3\times3}\cap L_{2,\mathrm{sym}}(\Omega)^{3\times3}\subseteq L_{2,\mathrm{sym}}(\Omega)^{3\times3} & \to L_{2}(\Omega)^{3}\\
\left(T_{ij}\right)_{i,j\in\{1,2,3\}} & \mapsto\left(\sum_{j=1}^{3}\partial_{j}T_{ij}\right)_{i\in\{1,2,3\}}.
\end{align*}
Furthermore, we define the operator $\Grad\coloneqq-\left(\Div_{c}\right)^{\ast}$
and $\Div\coloneqq-\left(\Grad_{c}\right)^{\ast}.$ 
\end{defn*}
The operator matrix $\left(\begin{array}{cc}
0 & \Div\\
\Grad & 0
\end{array}\right)$ is of the form discussed in Section 3, and thus, it suffices to show
that the boundary conditions \prettyref{eq:Dirichlet}-\prettyref{eq:frictional_bd_cd}
can be realized by a maximal monotone relation on $\mathcal{BD}(\Grad).$
This is the aim of the rest of this subsection. We first note that
Korn's inequality holds for Lipschitz-domains (see e.g. \cite{Ciarlet2005}),
which states that $H^{1}(|\Grad|+\i)$ and $H^{1}(|\grad|+\i)^{3}$
are isomorphic via the identity-mapping. Following the reasoning of
Section 4, we obtain that $\mathcal{BD}(\Grad)$ is isomorphic to
$H^{\frac{1}{2}}(\partial\Omega)^{3}$ and consequently, there exists
a continuous injection $\kappa:\mathcal{BD}(\Grad)\to L_{2}(\partial\Omega)^{3}$
with dense range. In this sense, the boundary condition \prettyref{eq:Dirichlet}
can be formulated as 
\begin{equation}
\pi_{L_{2}(\Gamma_{1})^{3}}\kappa\pi_{\mathcal{BD}(\Grad)}\partial_{0}u=f_{1},\label{eq:Dirichlet_2}
\end{equation}
which in particular implies $f_{1}\in\mathcal{R}\left(\pi_{L_{2}(\Gamma_{1})^{3}}\kappa\right).$
We now define a maximal monotone relation on $\left(L_{2}(\Gamma_{1})^{3}\right)^{\bot}=L_{2}(\Gamma_{2})^{3}\oplus L_{2}(\Gamma_{3})^{3}$,
which will represent the boundary conditions \prettyref{eq:Neumann}
and \prettyref{eq:frictional_bd_cd}.
\begin{defn*}
We denote by $\pi_{L_{2}(\Gamma_{2})^{3}}:\left(L_{2}(\Gamma_{1})^{3}\right)^{\bot}\to L_{2}(\Gamma_{2})^{3}$
and by $\pi_{L_{2}(\Gamma_{3})^{3}}:\left(L_{2}(\Gamma_{1})^{3}\right)^{\bot}\to L_{2}(\Gamma_{3})^{3}$
the orthogonal projections onto $L_{2}(\Gamma_{2})^{3}$ and $L_{2}(\Gamma_{3})^{3}$,
respectively. We define the relation $\tilde{g}\subseteq\left(L_{2}(\Gamma_{1})^{3}\right)^{\bot}\oplus\left(L_{2}(\Gamma_{1})^{3}\right)^{\bot}$
by 
\[
\tilde{g}\coloneqq\left\{ (x,y)\,|\,\pi_{L_{2}(\Gamma_{2})^{3}}y=f_{2},(\pi_{L_{2}(\Gamma_{3})^{3}}x,\pi_{L_{2}(\Gamma_{3})^{3}}y)\in g\right\} .
\]
\end{defn*}
\begin{lem}
\label{lem:rel_tilde_g}The relation $\tilde{g}$ is maximal monotone.
Moreover, if $g$ is bounded then so is $\tilde{g}.$\end{lem}
\begin{proof}
The maximal monotonicity follows by \prettyref{prop:dircet_sum_rel}
and the boundedness of $\tilde{g}$ in case of a bounded relation
$g$ is obvious. 
\end{proof}
In \cite{Migorski_2009} we find an assumption on $g$, which in particular
implies the boundedness of $g$. So, henceforth, we will assume that
$g$ is bounded. Moreover, we define the closed subspace $V$ of $\mathcal{BD}(\Grad)$
by 
\[
V\coloneqq\mathcal{N}\left(\pi_{L_{2}(\Gamma_{1})^{3}}\kappa\right),
\]
which consists of those elements in $\mathcal{BD}(\Grad),$ whose
trace is supported on $\Gamma_{2}\cup\Gamma_{3}$. 
\begin{lem}
\label{lem:maxmon_tilde_h}If $g$ is bounded, then $\tilde{h}\coloneqq\pi_{V}\kappa^{\ast}\tilde{g}\kappa\pi_{V}^{\ast}\subseteq V\oplus V$
is maximal monotone.\end{lem}
\begin{proof}
This follows by \prettyref{prop:sim_bd_rel} and \prettyref{lem:rel_tilde_g}.
\end{proof}
We are now able to define the realization $A\subseteq\left(\begin{array}{cc}
0 & \Div\\
\Grad & 0
\end{array}\right)$ subject to the boundary conditions \prettyref{eq:Dirichlet}-\prettyref{eq:frictional_bd_cd}.
\begin{thm}
The nonlinear operator $A\subseteq\left(\begin{array}{cc}
0 & \Div\\
\Grad & 0
\end{array}\right)$ with 
\begin{multline*}
\mathcal{D}(A)\coloneqq\bigg\{(v,-T)\in\mathcal{D}(\Grad)\times\mathcal{D}(\Div)\,\bigg|\,\pi_{L_{2}(\Gamma_{1})^{3}}\kappa\pi_{\mathcal{BD}(\Grad)}v=f_{1},\\
\left.\,\left(\pi_{V}\pi_{\mathcal{BD}(\Grad)}v,\pi_{V}\stackrel{\bullet}{\Div}\pi_{\mathcal{BD}(\Div)}\left(-T\right)\right)\in\tilde{h}\right\} 
\end{multline*}
is maximal monotone.\end{thm}
\begin{proof}
We first note that 
\[
\pi_{L_{2}(\Gamma_{1})^{3}}\kappa\pi_{V^{\bot}}^{\ast}:V^{\bot}\to\mathcal{R}(\kappa\pi_{L_{2}(\Gamma_{1})^{3}})
\]
is bijective according to the definition of $V$. Hence, $\pi_{L_{2}(\Gamma_{1})^{3}}\kappa\pi_{\mathcal{BD}(\Grad)}v=f_{1}$
is equivalent to 
\[
\pi_{V^{\bot}}\pi_{\mathcal{BD}(\Grad)}v=\tilde{f}_{1}\coloneqq\left(\pi_{L_{2}(\Gamma_{1})^{3}}\kappa\pi_{V^{\bot}}^{\ast}\right)^{-1}(f_{1}).
\]
Thus, defining 
\begin{equation}
h\coloneqq\left\{ (x,y)\in\mathcal{BD}(\Grad)\oplus\mathcal{BD}(\Grad)\,|\,\pi_{V^{\bot}}x=\tilde{f}_{1},(\pi_{V}x,\pi_{V}y)\in\tilde{h}\right\} ,\label{eq:h_frictional}
\end{equation}
we can write the domain of $A$ as 
\[
\mathcal{D}(A)=\left\{ (v,-T)\in\mathcal{D}(\Grad)\times\mathcal{D}(\Div)\,\left|\,\left(\pi_{\mathcal{BD}(\Grad)}v,\stackrel{\bullet}{\Div}\pi_{\mathcal{BD}(\Div)}\left(-T\right)\right)\in h\right.\right\} .
\]
Hence, $A$ is maximal monotone if $h$ is maximal monotone according
to \prettyref{thm:char_max_mon}. The latter follows by the maximal
monotonicity of $\tilde{h}$ (see \prettyref{lem:maxmon_tilde_h})
and \prettyref{prop:dircet_sum_rel}.
\end{proof}
In the remaining part of this subsection we discuss, in which sense
elements in the domain of $A$ satisfy the boundary conditions \prettyref{eq:Dirichlet}-\prettyref{eq:frictional_bd_cd}.
Following the rationale of Section 4, \prettyref{eq:Neumann} should
hold in the sense that 
\[
\langle(\Div-\Div_{c})\pi_{\mathcal{BD}(\Div)}^{\ast}\pi_{\mathcal{BD}(\Div)}(-T)|\pi_{\mathcal{BD}(\Grad)}^{\ast}u\rangle_{H^{-1}(|\Div|+\i),H^{1}(|\Grad|+\i)}=\langle f_{2}|\kappa u\rangle_{L_{2}(\partial\Omega)^{3}}
\]
for each $u\in\mathcal{BD}(\Grad)$ with $\kappa u\in L_{2}(\Gamma_{2})^{3}$
or, in other words, for each $u\in\mathcal{N}\left(\pi_{\left(L_{2}(\Gamma_{2})^{3}\right)^{\bot}}\kappa\right)\eqqcolon V_{2}.$
The latter gives 
\begin{align*}
 & \left\langle \left.\pi_{V_{2}}^{\ast}\pi_{V_{2}}\stackrel{\bullet}{\Div}\pi_{\mathcal{BD}(\Div)}(-T)\right|u\right\rangle _{\mathcal{BD}(\Grad)}\\
 & =\left\langle \left.\stackrel{\bullet}{\Div}\pi_{\mathcal{BD}(\Div)}(-T)\right|\pi_{V_{2}}^{\ast}\pi_{V_{2}}u\right\rangle _{\mathcal{BD}(\Grad)}\\
 & =\langle(\Div-\Div_{c})\pi_{\mathcal{BD}(\Div)}^{\ast}\pi_{\mathcal{BD}(\Div)}(-T)|\pi_{\mathcal{BD}(\Grad)}^{\ast}\pi_{V_{2}}^{\ast}\pi_{V_{2}}u\rangle_{H^{-1}(|\Div|+\i),H^{1}(|\Grad|+\i)}\\
 & =\langle f_{2}|\kappa\pi_{V_{2}}^{\ast}\pi_{V_{2}}u\rangle_{L_{2}(\partial\Omega)^{3}}\\
 & =\langle\pi_{V_{2}}^{\ast}\pi_{V_{2}}\kappa^{\ast}f_{2}|u\rangle_{\mathcal{BD}(\Grad)}
\end{align*}
for each $u\in\mathcal{BD}(\Grad).$ Thus, the appropriate formulation
for \prettyref{eq:Neumann} in our setting is 
\begin{equation}
\pi_{V_{2}}\stackrel{\bullet}{\Div}\pi_{\mathcal{BD}(\Div)}(-T)=\pi_{V_{2}}\kappa^{\ast}f_{2}.\label{eq:Neumann_2}
\end{equation}

Analogously \prettyref{eq:frictional_bd_cd} should hold in the sense
that there exists a function $f_{3}\in L_{2}(\Gamma_{3})^{3}$ such
that
\begin{equation}
(\pi_{L_{2}(\Gamma_{3})^{3}}\kappa\pi_{\mathcal{BD}(\Grad)}\partial_{0}u,f_{3})\in g\mbox{ and }\pi_{V_{3}}\stackrel{\bullet}{\Div}\pi_{\mathcal{BD}(\Div)}(-T)=\pi_{V_{3}}\kappa^{\ast}f_{3},\label{eq:frictional_2}
\end{equation}
where $V_{3}\coloneqq\mathcal{N}\left(\pi_{\left(L_{2}(\Gamma_{3})^{3}\right)^{\bot}}\kappa\right)$. 

Let now $\left(\pi_{\mathcal{BD}(\Grad)}\partial_{0}u,\stackrel{\bullet}{\Div}\pi_{\mathcal{BD}(\Div)}(-T)\right)\in h,$
where $h$ is given by \prettyref{eq:h_frictional}. Then 
\[
\pi_{V^{\bot}}\pi_{\mathcal{BD}(\Grad)}\partial_{0}u=\tilde{f}_{1},
\]
 which yields \prettyref{eq:Dirichlet_2}. Moreover, $\left(\pi_{V}\pi_{\mathcal{BD}(\Grad)}\partial_{0}u,\pi_{V}\stackrel{\bullet}{\Div}\pi_{\mathcal{BD}(\Div)}(-T)\right)\in\tilde{h},$
which implies the existence of an element $w\in\left(L_{2}(\Gamma_{1})^{3}\right)^{\bot}$
such that 
\[
(\kappa\pi_{V}^{\ast}\pi_{V}\pi_{\mathcal{BD}(\Grad)}\partial_{0}u,w)\in\tilde{g}
\]
and 
\[
\pi_{V}\stackrel{\bullet}{\Div}\pi_{\mathcal{BD}(\Div)}(-T)=\pi_{V}\kappa^{\ast}w.
\]
According to the definition of $\tilde{g},$ we get that 
\[
\pi_{L_{2}(\Gamma_{2})^{3}}w=f_{2}\mbox{ and }(\pi_{L_{2}(\Gamma_{3})^{3}}\kappa\pi_{V}^{\ast}\pi_{V}\pi_{\mathcal{BD}(\Grad)}\partial_{0}u,\pi_{L_{2}(\Gamma_{3})^{3}}w)\in g.
\]
Hence, we get 
\begin{align*}
\pi_{V_{2}}\stackrel{\bullet}{\Div}\pi_{\mathcal{BD}(\Div)}(-T) & =\pi_{V_{2}}\pi_{V}^{\ast}\pi_{V}\stackrel{\bullet}{\Div}\pi_{\mathcal{BD}(\Div)}(-T)\\
 & =\pi_{V_{2}}\pi_{V}^{\ast}\pi_{V}\kappa^{\ast}w\\
 & =\pi_{V_{2}}\kappa^{\ast}w\\
 & =\pi_{V_{2}}\kappa^{\ast}\pi_{L_{2}(\Gamma_{2})^{3}}w\\
 & =\pi_{V_{2}}\kappa^{\ast}f_{2},
\end{align*}
where we have used that $\pi_{V_{2}}=\pi_{V_{2}}\pi_{V}^{\ast}\pi_{V},$
since $V_{2}\subseteq V,$ and $\pi_{V_{2}}\kappa^{\ast}\pi_{\left(L_{2}(\Gamma_{2})^{3}\right)^{\bot}}=0$
by the definition of $V_{2}.$ This shows \prettyref{eq:Neumann_2}.
Analogously, one obtains 
\[
\pi_{V_{3}}\stackrel{\bullet}{\Div}\pi_{\mathcal{BD}(\Div)}(-T)=\pi_{V_{3}}\kappa^{\ast}\pi_{L_{2}(\Gamma_{3})^{3}}w,
\]
which yields \prettyref{eq:frictional_2} for $f_{3}\coloneqq\pi_{L_{2}(\Gamma_{3})^{3}}w$.

\section*{Acknowledgement}

The author thanks Birgit Jacob, who has initiated this study by a
question on a conference and Marcus Waurick for careful reading.

\end{document}